\documentclass[11pt, reqno,usenames, dvipsnames]{amsart}

\usepackage{microtype}%
\allowdisplaybreaks

\usepackage{enumitem}
\usepackage{times}
\usepackage{amsaddr}
\usepackage[margin=1in]{geometry}

\usepackage{url}
\usepackage{amssymb}
\usepackage{amsthm}
\usepackage[normalem]{ulem}
\usepackage{bbm}
\usepackage{array}
\usepackage[colorlinks=true,urlcolor=purple,
linkcolor=purple,citecolor=purple]{hyperref}
\usepackage{amsmath}
\usepackage{mathtools}

\usepackage{mathrsfs}
\usepackage{framed}
\usepackage{latexsym}
\usepackage{pb-diagram}
\usepackage[mathscr]{}
\usepackage{graphicx}
\usepackage[all]{xy} 
\usepackage{color}

\usepackage{algorithm}
\usepackage[noend]{algpseudocode}
\usepackage{thmtools,thm-restate} 

\makeatletter
\def\paragraph{\@startsection{paragraph}{4}%
  \z@\z@{-\fontdimen2\font}%
  {\normalfont\bfseries}}
\makeatother

\usepackage{tikz}
\usetikzlibrary{decorations.markings,backgrounds,calc,positioning,shapes, shadows,arrows,fit, automata}
\usepackage{tikz-cd}
\usetikzlibrary{arrows.meta}
\tikzset{>={Latex}}

\usepackage{xparse}
\NewDocumentCommand{\xrightarrows}{ O{}O{} }{%
\mathrel{%
\vcenter{\hbox{%
\begin{tikzpicture}
  \node[minimum width=1cm,minimum height=1ex,anchor=south,align=center] (a){\text{\vphantom{hg}#1}\\[0.5ex] \vphantom{hg}#2};
  \draw[<-,dashed] ([yshift=-0.35ex]a.west) -- ([yshift=-0.35ex]a.east);
  \draw[->] ([yshift=0.35ex]a.west) -- ([yshift=0.35ex]a.east);
\end{tikzpicture}}}}}

\newcommand{\N}{\mathbb{N}}

\newcommand{\Q}{\mathbb{Q}}
\newcommand{\R}{\mathbb{R}}
\newcommand{\C}{\mathbb{C}}

\newcommand{\mc}{\mathcal}

\renewcommand{\a}{\alpha}
\renewcommand{\b}{\beta}
\newcommand{\g}{\gamma}
\renewcommand{\d}{\delta}
\newcommand{\e}{\varepsilon}
\newcommand{\w}{\omega}
\newcommand{\s}{\sigma}
\newcommand{\ph}{\varphi}

\newcommand{\defeq}{\overset{\tiny\operatorname{def}}{=}}

\newcommand{\lp}{\left(}
\newcommand{\rp}{\right)}
\newcommand{\lbar}{\left|}
\newcommand{\rbar}{\right|}
\newcommand{\lnorm}{\left\|}
\newcommand{\rnorm}{\right\|}

\newcommand{\medlnorm}{\big\|}
\newcommand{\medrnorm}{\big\|}
\newcommand{\medlp}{\big(}
\newcommand{\medrp}{\big)}

\newcommand{\set}[1]{\left\{#1\right\}}

\renewcommand{\r}{\rightarrow}

\newcommand{\norm}[1]{\left\|#1 \right\|} 

\renewcommand{\emptyset}{\varnothing}

\newcommand{\inv}{^{-1}}

\newcommand{\offd}{\triangle}
\newcommand{\ond}{\blacktriangle}

\newcommand{\y}{\mathbf{y}}
\newcommand{\x}{\mathbf{x}}

\newcommand{\im}{\operatorname{im}}

\newcommand{\coup}{\mathscr{C}}

\newtheorem{theorem}{Theorem}
\newtheorem{question}{Question}
\newtheorem{proposition}[theorem]{Proposition}
\newtheorem{lemma}[theorem]{Lemma}
\newtheorem{corollary}[theorem]{Corollary}

\theoremstyle{definition}
\newtheorem{example}[theorem]{Example}
\newtheorem{remark}[theorem]{Remark}

\newtheorem{definition}[theorem]{Definition}
\newtheorem{claim}{Claim}

\newcommand{\dgm}{\mc{D}}
\newcommand{\dl}{d_p[l^q]}
\newcommand{\dlp}{d_p[l^2]}
\newcommand{\dlpq}{d_p[l^q]}
\newcommand{\dlt}{d_2[l^2]}
\newcommand{\dlpp}{d_p[l^p]}

\newcommand{\dlinf}{d_{\infty}[l^q]}
\newcommand{\dgml}{\dgm_p[l^q]}

\newcommand{\dgmlp}{\dgm_{p}[l^2]}
\newcommand{\dgmlt}{\dgm_{2}[l^2]}
\newcommand{\dgmlpq}{\dgm_{p}[l^q]}

\newcommand{\dgmlinf}{\dgm_{\infty}[l^q]}
\newcommand{\Deltapt}{\Delta^{\bullet}}
\newcommand{\pir}{\pi_{\R^2}}

\makeatletter
\newcommand{\pushright}[1]{\ifmeasuring@#1\else\omit\hfill$\displaystyle#1$\fi\ignorespaces}
\newcommand{\pushleft}[1]{\ifmeasuring@#1\else\omit$\displaystyle#1$\hfill\fi\ignorespaces}
\makeatother

\title{Geodesics in persistence diagram space}
\author{Samir Chowdhury}

\begin{document}

\date{\today}

\email{chowdhury.57@osu.edu}

\address{Department of Mathematics, The Ohio State University.}

\begin{abstract} 
It is known that for a variety of choices of metrics, including the standard bottleneck distance, the space of persistence diagrams admits geodesics. Typically these existence results produce geodesics that have the form of a convex combination. More specifically, given two persistence diagrams and a choice of metric, one obtains a bijection realizing the distance between the diagrams, and uses this bijection to linearly interpolate from one diagram to another. We prove that for several families of metrics, every geodesic in persistence diagram space arises as such a convex combination. For certain other choices of metrics, we explicitly construct infinite families of geodesics that cannot have this form. 
\end{abstract} 

\maketitle

\section{Introduction and problem statement}

A \emph{persistence diagram} or \emph{barcode} is a countable multiset of above-diagonal points in $\R^2$ along with the diagonal, which is counted with countably infinite multiplicity. We denote the collection of all possible diagrams by $\dgm$. Persistence diagrams were originally formulated as shape descriptors arising from applying \emph{persistent homology} to point cloud or metric datasets. In recent years, they have been generalized to the point where they can be studied as algebraic objects in their own right, without necessarily arising as a shape descriptor for a dataset. Relevant to this paper is the development of a variety of metrics on persistence diagrams with the overarching goal of defining Fr\'{e}chet means and related generalizations \cite{mileyko2011probability, turner2013means, tmmh, munch2015probabilistic}.

Persistence diagrams are typically compared using the \emph{bottleneck distance}, which is an $l^\infty$ matching distance where the matching cost is computed using an $l^\infty$ ground metric. In the aforementioned papers, the objects of study were variants of the bottleneck distance. Specifically, \cite{mileyko2011probability} considered $l^p$ matching for $p \in [1,\infty)$ with the $l^\infty$ ground metric, \cite{tmmh, munch2015probabilistic} considered $l^2$ matching with the $l^2$ (Euclidean) ground metric, and \cite{turner2013means} considered $l^p$ matching with an $l^p$ ground metric for $p \in [1,\infty]$.

By an overload of notation, let $\emptyset$ denote the \emph{empty diagram} consisting of just the diagonal with countably infinite multiplicity. In \cite{tmmh}, the authors defined a type of $l^2$ metric on $\dgm$ (denoted $d_2[l^2]$) and studied the space $\mc{D}_2[l^2] \defeq \{X \in \dgm : d_2[l^2](X,\emptyset) < \infty\}$. On this space, they characterized Fr\'{e}chet means and gave a procedure for computing these means. A necessary step for their constructions was a result showing that $(\mc{D}_2[l^2],d_2[l^2])$ is an \emph{Alexandrov space} with nonnegative curvature \cite[Theorem 2.5]{tmmh}. The proof of \cite[Theorem 2.5]{tmmh} in turn requires one to show that all geodesics in this space are of a convex combination form. Indeed, we show in Section \ref{sec:branch} that for certain other choices of metrics on $\dgm$, there exist geodesics which are not given by a convex combination form, and moreover there exist \emph{branching geodesics} which preclude a space from having nonnegative curvature in the sense of Alexandrov. Finally in Section \ref{sec:geod-char}, we show that for certain families of metrics, including the important case $p=q=2$, all geodesics are indeed of a convex combination form.

Our proof of this characterization result follows the strategy used by Sturm in proving an analogous result about geodesics in the space of metric measure spaces \cite{sturm2012space}. The existence results about branching geodesics and geodesics not given by a convex combination form are related to constructions we previously investigated in \cite{dgh-era}.

\medskip
\paragraph{Contributions} Following \cite{munch2015probabilistic}, we study persistence diagram metrics $\dlpq$ which involve an $l^p$ matching metric over an $l^q$ ground metric. For certain families of metrics, we show that $\dgm$ has geodesics that can be uniquely characterized as convex combinations. We are able to prove our result for the following families:
\begin{itemize}
\item $q = 2$, $p \in (1,\infty)$
\item $p = q \in [2,\infty)$.
\end{itemize}
We also provide counterexamples showing that geodesics are \emph{not} uniquely characterized in the cases $p = q =1$ and $p =\infty, \, q\in[1,\infty]$. 

Said differently: whereas it is easy to show that any optimal bijection yields a geodesic (via the convex-combination form), here we prove the harder reverse direction, i.e. that any geodesic arises as the convex-combination geodesic of an optimal bijection, at least for certain ranges of $p,q$. Furthermore, for certain other ranges of $p$ and $q$, we show that the negative result holds. So our focus is on the dashed line shown below.
\begin{align*}
\centering
\{\text{optimal bijection}\} \xrightarrows[\text{convex combination}] \{\text{geodesic}\} 
\end{align*}

\section{Definitions and statement of results}

Given sets $X,Y$ and an element $z \in X\times Y$, we write $\pi_X(z), \pi_Y(z)$ to denote the canonical projections of $z$ into $X$ and $Y$, respectively. 
The diagonal in $\R^2$ is denoted $\Delta :=	\{(x,x) \in \R^2 : x \in \R\}$. We also define $\Delta_\Q:= \{(x,x) \in \R^2 : x \in \Q\}$, i.e. the rational points on the diagonal. We write $\Delta^\infty$ or $\Delta_\Q^\infty$ to denote these sets counted with countably infinite multiplicity. The part of the plane above the diagonal is denoted $\R^2_>$, and the part of the plane above and including the diagonal is denoted $\R^2_{\geq}$. The $p$-norm in $\R^2$, for $p \in [1,\infty]$, is denoted $\lnorm \cdot \rnorm_p$. Given an above-diagonal point $x \in \R^2_>$, we write $\lnorm x - \Delta \rnorm_p$ to denote the perpendicular distance (in $p$-norm) between $x$ and the diagonal. We also write $\pi_{\Delta}(x)$ to denote the projection of $x$ onto the diagonal. When we suppress notation and write $\lnorm \cdot \rnorm$, we mean the Euclidean norm in $\R^2$. We will occasionally use the canonical identification between $\R^2$ and $\C$. 

The transpose of a vector $[v_1,v_2,\ldots,v_n]$ will be denoted $[v_1,v_2,\ldots,v_n]^T$.
Given an infinite-dimensional vector $V \in \R^\N$ and a function $f$ defined on each element of $V$, we will write $f(V)$ to denote $(f(v_1),f(v_2),\ldots ).$

\begin{definition}
\label{def:pdgm}
A \emph{persistence diagram} is a countable subset of $\R^2_> \times \N$ along with countably infinite copies of $\Delta$. 
This naming convention differs slightly from that of the standard persistence diagram (cf. \cite{tmmh}), which involves multisets in $\R^2$. However, we introduce the $\N$ coordinate so that different copies of the same point can be defined to occupy different entries in $\N$. We refer to the $\N$ component as the indexing component, and the $\R^2$ component as the geometric component. For a persistence diagram $X$, we let $X_>$ denote the above-diagonal portion of the diagram. The collection of all persistence diagrams is denoted $\dgm$. For any $x \in X$, the cardinality of $(\pi_{\R^2})\inv \circ \pi_{\R^2}(x)$ is the multiplicity of $\pi_{\R^2}(x)$. We write $m(x)$ to denote the multiplicity of $x$. 

Note that persistence diagrams are typically formulated as multisets, i.e. as a subset $Z\subseteq \R^2_{\geq}$ along with a multiplicity function $m:Z \r \N$. This multiset formulation can be recovered from the $\R^2_\geq \times \N$ formulation given above; the advantage of the above formulation is that it enables some of our later arguments involving convergence of sequences.

Crucially, given persistence diagrams $X,Y$ and points $x \in X,\, y\in Y$, we write $\norm{x - y}_p$ to mean $\norm{\pir(x) - \pir(y)}_p$. In other words, when computing distances between points in persistence diagrams, only the geometric component of each point is considered.
\end{definition}

\begin{example}
\label{ex:pers-dgm}
Let $X = \{(0,1,1)\} \cup \Delta^\infty$, $Y = \{(0,1,1), (0,1,2), (1,3,3)\} \cup \Delta^\infty$, and $Z= \{(0,1,2)\} \cup \Delta^\infty$. All three are persistence diagrams. Each of $X$ and $Z$ has a single off-diagonal point at $(0,1)$. $Y$ has an off-diagonal with multiplicity two at $(0,1)$, and another point with multiplicity one at $(1,3)$. For any of the metrics we later define, the distance between $X$ and $Z$ is zero. This is because their off-diagonal points only differ in the $\N$ coordinate, which is not relevant for the distances we consider. 
\end{example}

Let $X, Y \in \dgm$ be two persistence diagrams. We can always obtain bijections between $X$ and $Y$, matching points to the diagonal if needed. Next we introduce a family of $l^p$ matching distances which compute the expected cost of an optimal matching between $X$ and $Y$, where optimality is with respect to an $l^q$ ground metric. 
Given $p \in [1,\infty)$, $q \in [1,\infty]$, the \emph{$l^p[l^q]$ matching distance} between persistence diagrams is the function $\dlpq: \dgm \times \dgm \r [0,\infty]$ given by writing 
\[\dlpq(X,Y) \defeq \inf \set{\lp \sum_{x\in X}  \lnorm x - \ph(x)\rnorm_q^p \rp^{1/p}  
: \ph: X \r Y \text{ a bijection} } \text{ for any } X,Y \in \dgm.\]
For $p = \infty$, we have
\[\dlpq(X,Y) \defeq \inf \set{\sup_{x\in X}  \lnorm x - \ph(x)\rnorm_q 
: \ph: X \r Y \text{ a bijection} } \text{ for any } X,Y \in \dgm.\] 
A bijection $\ph$ for which the infimum above is attained is said to be \emph{optimal}.

\begin{remark} Here are special cases of the preceding definition.
\begin{itemize}
\item The \emph{bottleneck distance} corresponds to $p = \infty, \, q= \infty$.
\item The case $p \in [1,\infty)$ and $q = \infty$ was considered in \cite{mileyko2011probability}.
\item Both \cite{tmmh, munch2015probabilistic} considered the case $p=2, \, q=2$.
\item \cite{turner2013means} considered the case $p = q \in [1,\infty]$. 
\end{itemize}
\end{remark}

\begin{definition} For $p,q \in [1,\infty]$, the set $\{X \in \dgm : \dlpq(X,\emptyset) < \infty \}$ is denoted by $\dgmlpq$. Note that if $X \in \dgmlpq$ and $p < \infty$, then any open ball $U \subseteq \R^2_>$ separated from the diagonal by some $\e > 0$ can contain only finitely many points of $\pir(X)$.  
\end{definition}

\begin{remark} We make some simple but important remarks to guide the reader:
\begin{itemize}
\item Typically the persistence diagram is defined to be a multiset of points in the extended plane (including $\infty$). Note that our definition only allows for points on the plane, which is in keeping with the definition in \cite{tmmh}.
\item In \cite{tmmh,munch2015probabilistic}, the distance $\dlpq$ above is called the $l^q$-Wasserstein metric; we avoid this terminology because Wasserstein distances typically refer to distances between probability measures. 


\item A priori, $\dlpq$ is only a pseudometric on $\dgm$. To see this, let $A, B $ be two countable dense subsets of $[0,1]$ that are not equal. Write $X:= \{ (0,a,1) : a \in A\} \cup \Delta^\infty$ and $Y:= \{(0,b,1) : b\in B\} \cup \Delta^\infty$. Then $\dlpq(X,Y) = 0$, even though $X \neq Y$. 
\end{itemize}
\end{remark}

A \emph{curve} in $(\dgmlpq,\dlpq)$  is a continuous map $\g:[0,1] \r \dgmlpq$. Such a curve is called a \emph{geodesic} \cite[Section I.1]{bridson2011metric} if for any $s, t \in [0,1]$, 
\[\dlpq(\g(s),\g(t)) = |t-s|\cdot \dlpq(\g(0),\g(1)).\]
In \cite{tmmh}, the authors gave a constructive proof showing that $(\dgmlt,\dlt)$ is a \emph{geodesic space}, i.e. that for any $X, Y \in \dgmlt$, there exists a geodesic from $X$ to $Y$. As a precursor to the construction, they first proved the following result showing that between any $X, Y \in \dgmlt$, there exists an optimal bijection $\ph$ realizing the infimum in the definition of $\dlt(X,Y)$:

\begin{theorem}[Existence of optimal bijections, \cite{tmmh} Proposition 2.3]
\label{thm:existence}
Let $X, Y \in \dgmlt$. Then there exists a bijection $\ph: X \r Y$ such that $\sum_{x \in X} \norm{x - \ph(x)}^2 = \dlt(X,Y)^2$. 
\end{theorem}

The construction of the geodesics is as follows: given any $X, Y \in \dgmlpq$, let $\ph$ be an optimal bijection. For the time being, we ignore the $\N$ coordinates of the persistence diagrams. Write $\g(0) \defeq X$, $\g(1) \defeq Y$, and for any $t \in (0,1)$, 
\[\g(t) \defeq \set{ (1-t)x + t\ph(x) : x \in X}.\]

Regardless of the choice of $p,q \in [1,\infty]$, such a curve defines a geodesic (cf. Corollary \ref{cor:geod-exist}). Note that different choices of $p,q$ may lead to different bijections $\ph$ being optimal. We call any geodesic of this form a \emph{convex-combination geodesic}. Conversely, we refer to geodesics \emph{not} of this form as \emph{deviant} geodesics.
Returning to the question of dealing with the indexing coordinate $\N$: recall that $\dlpq$ is blind to this coordinate, so we can define the convex-combination geodesic $\gamma$ in the following manner and still maintain continuity:
\[ \g(t) \defeq \set{ [(1-t)x_1 + t\ph(x)_1, (1-t)x_2 + t\ph(x)_2, x_3]^T : x=[x_1,x_2,x_3]^T \in X} \text{ for } t\in [0,1),\]
and $\g(1) \defeq Y$. In other words, the indexing coordinate stays constant for $t \in [0,1)$, and switches to the appropriate coordinate at $t =1$.

We will occasionally discuss \emph{branching geodesics}. A geodesic $\g:[0,1] \r \dgmlpq$ \emph{branches} at $t_0 \in (0,1)$ if there exists a geodesic $\widetilde{\g}: [0,1] \r \dgmlpq$ such that $\widetilde{\g}$ agrees with $\g$ on $[0,t_0]$, and is distinct from $\g$ on $(t_0,t_0+\e]$ for some $\e > 0$.

With this terminology, we now pose the main question motivating this paper.

\begin{question} 
\label{q:convex-comb}
For which pairs $(p,q)$ can we say that all geodesics in $\dgmlpq$ have the form of convex-combination geodesics?
\end{question}

Our first result shows that setting $p=\infty$ simultaneously produces branching and deviant geodesics in $\dgmlpq$. In particular, the existence of branching geodesics implies that $(\dgmlpq,\dlpq)$ for $q \in [1,\infty]$, $p = \infty$ cannot have nonnegative Alexandrov curvature (\cite[Chapter 10]{burago}).

\begin{theorem} 
\label{thm:branch-dev-infty}
Let $p = \infty, \, q \in [1,\infty]$. There exist infinite families of both branching and deviant geodesics in $(\dgmlpq,\dlpq)$.
\end{theorem}

\begin{remark}
\cite{turner2013means} showed---via a direct examination of an inequality characterizing Alexandrov curvature---that in the case $p = q \in [1,2) \cup (2,\infty]$, $\dgmlpq$ does not have nonnegative Alexandrov curvature. 
\end{remark}

We collect another related result for the case $p =q =1$: 
\begin{theorem}
\label{thm:branch-dev-one}
Let $q = p =1$. There exist infinite families of branching and deviant geodesics in $(\dgmlpq,\dlpq)$.
\end{theorem}

Theorems \ref{thm:branch-dev-infty} and \ref{thm:branch-dev-one} serve to make Question \ref{q:convex-comb} more interesting. The next result is the finite version of our answer to Question \ref{q:convex-comb}.

\begin{theorem}[Characterization of geodesics I]
\label{thm:char-I}
Fix $p,q$ in the following ranges:
\begin{itemize}
\item $p = q \in [2,\infty),$
\item $q = 2$, $p \in (1,\infty).$
\end{itemize}
Let $X, Y \in \dgmlpq$ be diagrams having finitely many points outside the diagonal, and let $\mu:[0,1] \r \dgmlpq$ be a geodesic from $X$ to $Y$. Then there exists a convex-combination geodesic $\g:[0,1] \r \dgmlpq$ from $X$ to $Y$ such that for each $t \in [0,1]$, we have $\dlpq(\g(t),\mu(t)) = 0$. 
\end{theorem}

This result in fact generalizes to the setting of countably-many off-diagonal points.

\begin{theorem}[Characterization of geodesics II]
\label{thm:char-II}
Fix $p,q$ in the following ranges:
\begin{itemize}
\item $p = q \in [2,\infty),$
\item $q = 2$, $p \in (1,\infty).$
\end{itemize}
Let $X, Y \in \dgmlpq$, and let $\mu:[0,1] \r \dgmlpq$ be a geodesic from $X$ to $Y$. Then there exists a convex-combination geodesic $\g:[0,1] \r \dgmlpq$ from $X$ to $Y$ such that for each $t \in [0,1]$, we have $\dlpq(\g(t),\mu(t)) = 0$. 
\end{theorem}

These characterization theorems have the following interpretation. Suppose we are given $X, Y \in \dgmlpq$ for the specified choices of $p$ and $q$ and a geodesic $\mu$ from $X$ to $Y$. Then $\mu$ is a convex-combination geodesic for some optimal bijection $\ph:X \r Y$. Furthermore, for each $x \in X$, we obtain a straight-line path $\gamma_x$ from $x$ to $\ph(x)$. By the construction of convex-combination geodesics, any optimal bijection produces a geodesic; these theorems assert the \emph{inverse} result that any geodesic comes from an optimal bijection, at least for the prescribed choices of $p$ and $q$.

\subsection{Recasting diagram metrics as $l^p$ norms and OT problems}
\label{sec:OT}

A key property of persistence diagrams is that the diagonal is counted with infinite multiplicity; this geometric trick ensures that bijections are always possible, and hence the $\dlpq$-type distances are always defined. While a persistence diagram contains uncountably many points according to Definition \ref{def:pdgm}, only countably many points are actually ever involved in computing a $\dlpq$-type distance. Specifically, we can view $\dlp$ as an $l^p$ norm. To see this, let $X,Y \in \dgmlp$. Recall that $X_>, Y_>$ consists of the (countably many) off-diagonal points of $X$ and $Y$. Define the $^*$ operation as the following:
\begin{align*}
X^* := X_> \cup \{ \pi_{\Delta}(y) : y \in Y_>\} \cup \Delta_\Q^\infty, \qquad \text{for } X \in \dgmlpq.
\end{align*}
The multiset $X^*$ consists of the off-diagonal points of $X$, a copy of the diagonal projection for each off-diagonal point of $Y$, and the rational diagonal points counted with countably infinite multiplicity. To ease the notation, we did not specify the indexing coordinates for the points in $\{ \pi_{\Delta}(y) : y \in Y_>\}$, but it is to be understood that the indices are chosen such that multiple off-diagonal points with the same diagonal projection are mapped to different slots in $\N$. The idea of including the rationals on the diagonal is the following: the set in the middle contains redundancies, so when obtaining matchings, it may be the case that the redundant diagonal points in $X_> \cup \{ \pi_{\Delta}(y) : y \in Y_>\}$ have to get matched to diagonal points in $Y_> \cup \{ \pi_{\Delta}(x) : x \in X_>\}$. By including rational points on the diagonal, we ensure (by the density of the rationals) that this matching of diagonal points contributes zero cost.

In particular, $X^*$ is a countable set (perhaps invoking the axiom of countable choice as necessary). Fix an enumeration $X^*= \{x^1,x^2,\ldots \}$. Then we think of $X^*$ as the map $X^*:\N \r \C$ given by $i \mapsto x_i  \mapsto \pir(x_i)$. 
Next define $Y^*$ analogously, and consider any bijection $\ph:X^* \r Y^*$. We again treat $\ph(X^*)$ as an infinite-dimensional vector, i.e, a map $\ph(x): \N \r \C$ given by $i\mapsto \ph(x_i)  \mapsto \pir(\ph(x_i))$. Here we are using the canonical identification of $\C$ with $\R^2$. 

Next we introduce some cost functions. Let $p,q \in [1,\infty]$, and let $X,Y \in \dgmlpq$. Define the following functional for a bijection $\ph:X \r Y$:
\begin{align}
C_p[l^q](\ph):= \begin{cases}
\lp \sum_{x\in X} \norm{x - \ph(x)}_q^p \rp^{1/p} &: p \in [1,\infty)\\
\sup_{x\in X} \norm{x - \ph(x)}_q &: p = \infty.
\end{cases}
\label{eq:Cp-defn}
\end{align} 
When $q = 2$, we reduce notation and simply write $C_p$ instead of $C_p[l^2]$. 

For the next few definitions, we fix $q=2$ and consider $X,Y \in \dgmlp$. Now for $p \in [1,\infty)$, consider the functional
\begin{align}
J_p(\ph) := \norm{ X^* - \ph(X^*) }_{l^p},
\label{eq:dgm-norm}
\end{align}
where the $l^p$-norm is given as
\begin{align*}
\norm{X^* - \ph(X^*) }_{l^p} := \medlp \sum_{i \in \N} | x_i - \ph(x_i) |^p \medrp^{1/p} = \medlp \sum_{i \in \N} \norm{ x_i - \ph(x_i) }_2^p \medrp^{1/p}
\end{align*}
if the sum converges, and as $\infty$ otherwise. Note that by our choice of $X,Y \in \dgmlp$, there always exists $\ph$ such that the preceding sum converges. For such $\ph$, the vector $X^* - \ph(X^*)$ belongs to $l^p$.
Here also recall from Definition \ref{def:pdgm} that $\norm{x - y}_p = \norm{ \pir(x) - \pir(y)}_p$. Each summand is an absolute value, i.e. a Euclidean norm, that is raised to the $p$th power. 
The $l^\infty$ norm is likewise defined as 
\[ \norm{x - \ph(x) }_{l^\infty} : = \sup_{i\in \N} \norm{x_i - \ph(x_i)}_2.\]

\begin{definition}
For any bijection $\ph':X \r Y$, define $\Lambda_{\ph'}$ to be the collection of bijections $\ph:X^* \r Y^*$ agreeing with $\ph'$ on off-diagonal points of $X$ and $Y$. 
\end{definition}

By the construction of $X^*, Y^*$ and the density of the rationals, we have 
\begin{align}
C_p(\ph')  = \medlp \sum_{x \in X} \norm{x - \ph'(x)}_2^p \medrp^{1/p} = \inf \{ J_p(\ph) : \ph \in \Lambda_{\ph'} \} \, \text{for any bijection $\ph':X \r Y$}.
\label{eq:bijection-approx}
\end{align}
In particular, the matching cost of $\ph \in \Lambda_{\ph'}$ differs from that of $\ph'$ only in how it produces a matching among the ``redundant" points on the diagonal.

Finally we observe that for all $p \in [1,\infty]$,
\[\dlp(X,Y) = \inf \{C_p(\ph) : \ph:X \r Y \text{ a bijection} \} = \inf \{J_p(\ph) : \ph:X^* \r Y^* \text{ a bijection} \}.\]
Here we are using the following observations: (1) any $\ph$ infimizing $C_p$ does not move diagonal points unnecessarily, and (2) any $\ph$ infimizing $J_p$ agrees with a $C_p$ infimizer on off-diagonal points and incurs zero cost for infinitesimally ``sliding" points along the diagonal.

\begin{remark} The distinction between $J_p$ and $C_p$ is that $J_p$ is an $l^p$ norm. This reformulation allows us to use powerful $l^p$ space inequalities to produce results for $\dlp$. It is not clear to us if this approach can be extended to $\dlpq$ for $q \neq 2$; attempting to prove one of the inequalities we need (Clarkson's inequality, Lemma \ref{lem:clarkson}) with $q \neq 2$ leads to some difficulty.  
\end{remark}

At least in the case of diagrams having finitely many off-diagonal points, one could similarly reformulate a $\dlp$ distance as an \emph{optimal transportation} (OT) problem. This idea is used below, where we describe a method (\cite{lacombe2018large}) for recasting the computation of a diagram metric as an OT problem. 

Given appropriately defined measures $\mu, \nu$ on measure spaces $X$ and $Y$, we write $\coup(\mu,\nu)$ to denote the collection of all \emph{coupling measures}, i.e. measures $\g$ on $X\times Y$ with marginals $\mu$ and $\nu$.

\begin{definition}
Following \cite{lacombe2018large}, we let $\Deltapt$ denote a virtual point representing the diagonal. We also use the notation $\R^{2\bullet}:= \R^2 \cup \{\Deltapt\}$ (and resp. $\R^{2\bullet}_>:= \R^2_> \cup \{\Deltapt\}$). For $x \in \R^2$, we use the notation $\lnorm x - \Deltapt \rnorm_p$ to denote $\lnorm x - \pi_{\Delta}(x) \rnorm_p$. We also set $\lnorm \Deltapt - \Deltapt \rnorm_p = 0$.  
\end{definition}

Let $X, Y \in \dgmlp$, $p \in [1,\infty]$, be diagrams having finitely many off-diagonal points. Let $n_X := |X_>|, \, n_Y:= |Y_>|$, and set $n:= n_X + n_Y$. Then we define:
\begin{align*}
X^\bullet &:= X_> \cup \{ (\Deltapt, j) : 1\leq j \leq n_Y \} \subseteq \R^{2\bullet}_> \times \N, \\
Y^\bullet &:= Y_> \cup \{ (\Deltapt, j) : 1\leq j \leq n_X \} \subseteq \R^{2\bullet}_> \times \N. 
\end{align*}

Then $n = |X^\bullet| = |Y^\bullet|$. Given arbitrary measures $\mu_{X^\bullet},\mu_{Y^\bullet}$ on $X^\bullet$ and $Y^\bullet$, respectively, and a coupling measure $\g \in \coup(\mu_{X^\bullet},\mu_{Y^\bullet})$ (i.e. a transport plan), the $L^p[l^2]$ transport cost is defined as:
\begin{align*} 
T_p(\g) :=  \lnorm \pi_{X^\bullet} - \pi_{Y^\bullet} \rnorm_{L^p(\g)}  
&=  \lp \int_{X^\bullet\times Y^\bullet} | \pi_{X^\bullet}(x,y) - \pi_{Y^\bullet}(x,y) |^p \,d\g(x,y) \rp^{1/p}\\
&= \lp \sum_{i,j} \lnorm x_i - y_j \rnorm_2^p \g(x_i,y_j) \rp^{1/p}.
\end{align*} 
Here $\pi_{X^\bullet} : X^\bullet\times Y^\bullet \r X^\bullet ,  \pi_{Y^\bullet}:X^\bullet\times Y^\bullet \r Y^\bullet$ are the canonical projection maps. More specifically, by taking the canonical identification of $\R^2$ with $\C$, these are maps $X^\bullet \times Y^\bullet \r \C$, so we are able to view them as maps in the $L^p$ space of complex-valued measurable functions. Measurability holds because these maps, being defined on discrete spaces, are trivially continuous. The absolute value in the integrand is taken for complex numbers, i.e. it corresponds to the Euclidean norm. The $l^2$ ground norm is the canonical choice when working over an $L^p$ space.

Next let $\mu_{X^\bullet}:= \sum_{i=1}^n \d_{x_i}$ denote the uniform measure on $X^\bullet$, and similarly let $\mu_{Y^\bullet}$ denote the uniform measure on $Y^\bullet$. Then we have (see also \cite[Proposition 1]{lacombe2018large}):

\begin{proposition}
\label{prop:OT}
Given $X,Y, X^\bullet,Y^\bullet$ as above, we have the following identity:
\[ \dlp(X,Y) =  \inf_{\g \in \coup(\mu_{X^\bullet},\mu_{Y^\bullet})} T_p(\g).\]
\end{proposition}

\begin{proof}
It is well-known as a consequence of Birkhoff's theorem (see \cite[\S0.1]{villani2003topics}) that the OT cost between measures on $n$-point spaces giving equal mass to all points is realized by a coupling that can be represented as an $n\times n$ permutation matrix. This permutation $\s$ provides the bijection in the definition of $\dlp(X,Y)$. 
\end{proof}

\begin{remark} The preceding OT formulation appears to work only in the case of diagrams with finitely many off-diagonal points. It would be interesting to clarify if a $\dlp$ distance between diagrams having countably many off-diagonal points can be formulated as an OT problem. The difficulty arises from ensuring that the optimal transportation plans correspond to permutation matrices, as required for the bijections in the definition of $\dlp$. 
\end{remark}

\section{Branching and deviant geodesics}
\label{sec:branch}

We now proceed to the proofs of Theorems \ref{thm:branch-dev-infty} and \ref{thm:branch-dev-one}.

\begin{figure}
\begin{minipage}{0.45\linewidth}
\centering
\begin{tikzpicture}
[node distance = 0pt, X/.style={circle,draw=purple,fill=purple,thick,inner sep = 0pt,minimum size = 6pt},
Y/.style={rectangle,draw=NavyBlue,fill=NavyBlue,thick,inner sep = 0pt,minimum size = 6pt}]
\draw[->,ultra thick] (0,0)--(5,0) ;
\draw[->,ultra thick] (0,0)--(0,5) ;
\draw[-,ultra thick] (0,0)--(4.5,4.5);
\node[X] (x) at (0,4.5)  {} ;
\node[X] (x') at (0,2.5)  {};
\node (mu) at (1,4)  {$\mu$};
\draw[-,thick, purple, dashed] (0,2.5)--(1.25,1.25);
\draw[-,thick, purple, dashed] (0,4.5)--(2.25,2.25);
\node [below left=of x] {$k$};
\node [below left=of x'] {$j$};
\end{tikzpicture}
\end{minipage}
\begin{minipage}{0.45\linewidth}
\centering
\begin{tikzpicture}
[node distance = 0pt, X/.style={circle,draw=purple,fill=purple,thick,inner sep = 0pt,minimum size = 6pt},
Y/.style={rectangle,draw=NavyBlue,fill=NavyBlue,thick,inner sep = 0pt,minimum size = 6pt}]
\draw[->,ultra thick] (0,0)--(5,0) ;
\draw[->,ultra thick] (0,0)--(0,5) ;
\draw[-,ultra thick] (0,0)--(4.5,4.5);
\node[Y] (y) at (0,4.5)  {};
\node[Y] (y') at (0,1)  {};
\node (nu) at (1,4)  {$\nu$};
\draw[-,thick, NavyBlue, dashed] (0,1)--(0.5,0.5);
\draw[-,thick, NavyBlue, dashed] (0,4.5)--(2.25,2.25);
\node [below left=of y] {$k$};
\node [below left=of y'] {$l$};
\end{tikzpicture}
\end{minipage}
\caption{Diagrams $Y$ (left) and $Z$ (right) as defined in the proof of Theorem \ref{thm:branch-dev-infty}. The separation between $j$ and $k$ is not to scale; in the proof we require $k > 3j$. }
\end{figure}
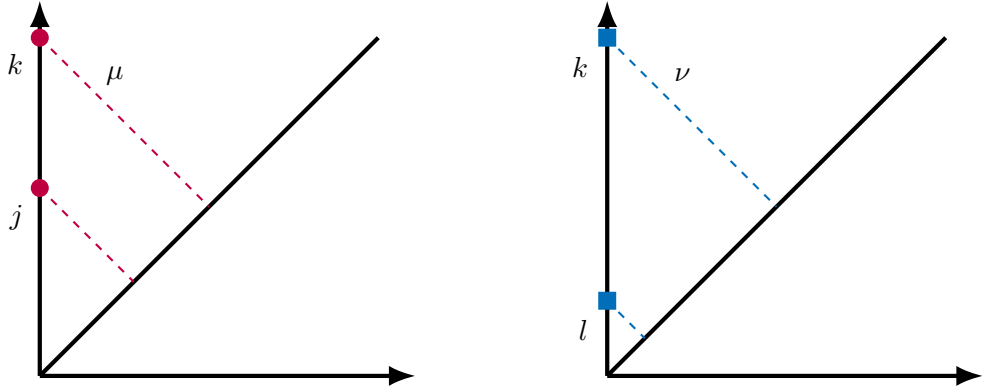

\begin{proof}[Proof of Theorem \ref{thm:branch-dev-infty}]

We begin with the proof of branching geodesics.
Let $X = \emptyset$, $Y = \{(0,k),(0,j)\}$, and $Z = \{(0,k),(0,l)\}$ for $0 < l  < j < 3j < k$. Now we define two curves $\mu, \nu: [0,1] \r \dgmlinf$ as follows:
\begin{align*}
\mu(t) &\defeq 
{
\begin{cases}
\set{\lp \frac{k}{2}t, \frac{k}{2}(2-t)\rp, \lp \frac{j}{2}(3t), \frac{j}{2}(2-3t) \rp } & 0 \leq t \leq \frac{1}{3}\\
\set{ \lp \frac{k}{2}t, \frac{k}{2}(2-t) \rp } & \frac{1}{3} \leq t \leq 1.
\end{cases} 
}\\
\nu(t) &\defeq 
{
\begin{cases}
\set{\lp \frac{k}{2}t, \frac{k}{2}(2-t)\rp, \lp \frac{l}{2}(3t), \frac{l}{2}(2-3t) \rp } & 0 \leq t \leq \frac{1}{3}\\
\set{ \lp \frac{k}{2}t, \frac{k}{2}(2-t) \rp } & \frac{1}{3} \leq t \leq 1.
\end{cases} 
}
\end{align*}

Thus $\mu, \nu$ are curves from $Y, Z$ to $X$. For convenience, define 
\[\mathbf{k}(t)\defeq \lp \tfrac{k}{2}t, \tfrac{k}{2}(2-t)\rp \text{ for }t \in [0,1], \qquad \mathbf{j}(t) \defeq \lp \tfrac{j}{2}(3t), \tfrac{j}{2}(2-3t) \rp \text{ for } t \in [0,\tfrac{1}{3}].\]

We check that $\mu,\nu$ are geodesics. It suffices to show this for $\mu$. First we see that $\dlinf(X,Y)$ is the perpendicular $q$-norm distance from $(0,k)$ to the diagonal; this is just $2^{(1/q) - 1}k$.

Let $s,t \in [\frac{1}{3},1]$. We observe that an optimal bijection matches $\mathbf{k}(s)$ and $\mathbf{k}(t)$; hence we have: 
\[\dlinf(\mu(s),\mu(t)) = \lnorm \lp \tfrac{k}{2}t, \tfrac{k}{2}(2-t)\rp - \lp \tfrac{k}{2}s, \tfrac{k}{2}(2-s)\rp \rnorm_q = 2^{(1/q)-1}k\lbar t - s \rbar = \lbar t- s \rbar \dlinf(X,Y).\] 

Let $s,t \in [0,\frac{1}{3}]$. First we claim that $\dlinf(\mu(s),\mu(t))$ is realized by the $q$-norm distance between $\mathbf{k}(s)$ and $\mathbf{k}(t)$. By the previous work, this is just $2^{(1/q)-1}k\lbar t - s \rbar$. We compare this to $\lnorm \mathbf{j}(s) - \mathbf{j}(t)\rnorm_q$:
\[\lnorm 
\lp \tfrac{j}{2}(3t), \tfrac{j}{2}(2-3t) \rp
-
\lp \tfrac{j}{2}(3s), \tfrac{j}{2}(2-3s) \rp
\rnorm_q 
= 
(3j)|t-s|2^{(1/q)-1}
< k|t-s|2^{(1/q)-1},
\]
where the last inequality holds because $3j < k$ by assumption. Thus $\dlinf(\mu(s),\mu(t)) = \lbar t- s \rbar \dlinf(X,Y).$ Notice that in this computation, it was implicit that an optimal matching would match $\mathbf{k}(s)$ to $\mathbf{k}(t)$ and $\mathbf{j}(s)$ to $\mathbf{j}(t)$; a cross-matching would not be optimal due to the greater distance that would need to be traversed. 

Finally let $s \in [0,\frac{1}{3}], t \in (\frac{1}{3},1]$. Again we claim that $\dlinf(\mu(s),\mu(t))$ is realized by $\lnorm \mathbf{k}(s) - \mathbf{k}(t) \rnorm_q$. The previous work shows that $\lnorm \mathbf{k}(s) - \mathbf{k}(t) \rnorm_q > \medlnorm \mathbf{j}(s) - (\tfrac{j}{2},\tfrac{j}{2}) \medrnorm_q$. It follows that $\dlinf(\mu(s),\mu(t)) = \lbar t- s \rbar \dlinf(X,Y).$

This shows that $\mu$ is a geodesic. The proof for $\nu$ is analogous. So $\mu, \nu$ are geodesics which are equal on $[\tfrac{1}{3},1]$, but clearly they branch at $t= \tfrac{1}{3}$ since $\dlinf(Y,Z) > 0$. Since $l < j < 3j < k$ were arbitrary, there are in fact infinitely many such branching geodesics. This concludes the first part of the proof.  \hfill$\blacksquare$

Notice that $\mu, \nu$ are not convex-combination geodesics; the points at $(0,j)$ and $(0,l)$ move too fast for the geodesics to be convex-combination, but slow enough that the geodesic property still holds. Even though these are deviant geodesics, there still seem to be bijections providing straight lines for the points to interpolate through. However, this need not be the case, and deviant geodesics may exist even when there is no supporting bijection. We see such a construction next.

Let $W = \{(0,k)\}$. Now we define a curve $\w:[0,1] \r \dgmlpq$ as follows:
\begin{align*}
\w(t) &\defeq 
{
\begin{cases}
\set{\lp \frac{k}{2}t, \frac{k}{2}(2-t)\rp, \lp j(\frac{1}{2} - t) , j(\frac{1}{2} + t) \rp } & 0 \leq t \leq \frac{1}{2}\\
\set{ \lp \frac{k}{2}t, \frac{k}{2}(2-t) \rp,  \lp j(t - \frac{1}{2}) , j(\frac{3}{2} - t) \rp } & \frac{1}{2} \leq t \leq 1.
\end{cases} 
}
\end{align*}
Then $\w$ is a curve from $W$ to $X$. Note that $\w(0)$ contains one off-diagonal point $(0,k)$, and this point linearly moves to the diagonal as $t \uparrow 1$. However, starting at $t = 0$, a point emerges from the diagonal at $(j/2,j/2)$ and moves linearly to $(0,j)$ as $t\uparrow 1/2$, which then returns to the diagonal as $t \uparrow 1$. Calculations such as the ones carried out above show that $\w$ is a geodesic; for the reader's convenience, we note that the point moving back and forth between $(0,j)$ and the diagonal has speed $j < k/3,$ so the $l^\infty$ matching only sees the $q$-norm distance between $\mathbf{k}(s)$ and $\mathbf{k}(t)$. This is the reason $\w$ is a geodesic. However, $\w$ is not a convex-combination geodesic from $W$ to $X$. Moreover, for different choices of $j$, we get infinitely many geodesics from $W$ to $X$, all of which are mutually distinct. This concludes the proof. \qedhere

\end{proof}

Next we proceed to the proof of Theorem \ref{thm:branch-dev-one}.

\begin{proof}[Proof of Theorem \ref{thm:branch-dev-one}]
Fix $k \gg 0$ so that we do not have to consider situations where points are matched to the diagonal. Let $X = \{(0,k), (1,k-1)\}$ and $Y = \{(1,k+1),(2,k)\}$. This configuration is illustrated in Figure \ref{fig:deviant11}. Define a curve $\mu:[0,1] \r \dgmlpq$ as follows:

\begin{align*}
\mu(t) &\defeq 
{
\begin{cases}
\set{\lp 2t,k  \rp, \lp 1, k-1+2t \rp } & 0 \leq t \leq \frac{1}{2}\\
\set{ \lp 1,k+2(t- \frac{1}{2}) \rp,  \lp 1 + 2(t-\frac{1}{2}), k \rp } & \frac{1}{2} \leq t \leq 1.
\end{cases} 
}
\end{align*}
This curve corresponds to the lefmost configuration in Figure \ref{fig:deviant11}. The points $x,x'$ come together at the center, then bend and travel to $y,y'$, respectively. Next we verify that $\mu$ is a geodesic. First note that $d_1[l^1](X,Y) = 4$. Next let $s\leq t \in [0,\frac{1}{2}].$ The optimal matching between $\mu(s)$ and $\mu(t)$ happens in the simple way: points along the dashed line get matched, and points along the solid line get matched (here we are referring to Figure \ref{fig:deviant11}). The cost of this matching is as follows:
\begin{align*}
d_1[l^1](\mu(s),\mu(t)) &= \lnorm (2s,k) - (2t,k) \rnorm_1 + \lnorm (1,k-1+2s) - (1,k-1+2t) \rnorm_1\\
&= 2(t-s) + 2(t-s) = |t-s| \, d_1[l^1](X,Y).
\end{align*}

The verification for $s,t \in [\frac{1}{2},1]$ is analogous. An interesting case is $s \in [1,\frac{1}{2}),\, t \in [\frac{1}{2},1]$. By virtue of using the $l^1$ ground metric, there are two optimal bijections: matching the points according to the dashed/solid lines, and cross-matching points on the dashed and solid lines. Using the first of these bijections, we calculate:
\begin{align*}
d_1[l^1](\mu(s),\mu(t)) &= \medlnorm (2s,k) - (1,k+2(t-\tfrac{1}{2})) \medrnorm_1 + \medlnorm (1,k-1+2s) - (1+2(t-\tfrac{1}{2}),k) \medrnorm_1\\
&= 2(t-s) + 2(t-s) = |t-s| \, d_1[l^1](X,Y).
\end{align*}
Thus $\mu$ is a geodesic. Note that it is different from the convex-combination geodesic illustrated at the right of Figure \ref{fig:deviant11}. 

Moreover, note that curves with corners, as illustrated in the middle of Figure \ref{fig:deviant11}, would also be geodesics by virtue of the ground metric being $l^1$. There is an infinite choice of positions for these corners, and so we get an infinite family of deviant geodesics which are all distinct from each other. \hfill $\blacksquare$

Now we proceed to the proof of branching geodesics. We refer the reader to Figure \ref{fig:branch11}. Starting with $X = \{(0,k), (1,k-1)\}$ as before and a fixed $r \in [0,1]$, consider the curve $\nu_r$ which: (1) transports the points $x = (0,k)$ and $x' = (1,k-1)$ to $(1,k)$ at constant speed over the interval $t \in [0,\frac{1}{2}]$, and (2) moves $x,x'$ jointly to $(1,k+r)$ and then to $(1 + (1-r), k+r)$, all at constant speed over the interval $t \in [\frac{1}{2},1]$. The cases $r = 1, 0, 0.5$ are illustrated from left to right, respectively, in Figure \ref{fig:branch11}. Calculations analogous to the ones carried out above show that these curves are all geodesics, and by construction, they branch at $t = \frac{1}{2}$. Thus $\{\nu_r : r \in [0,1]$ is an infinite family of branching geodesics in $\dgmlpq$.
\end{proof}

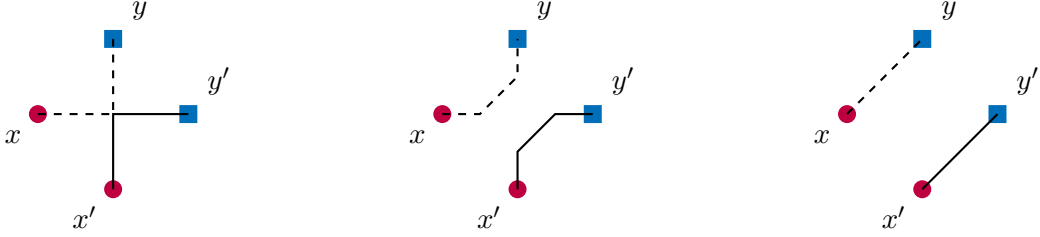
\begin{figure}
\begin{minipage}{0.32\linewidth}
\centering
\begin{tikzpicture}
[node distance = 0pt, X/.style={circle,draw=purple,fill=purple,thick,inner sep = 0pt,minimum size = 6pt},
Y/.style={rectangle,draw=NavyBlue,fill=NavyBlue,thick,inner sep = 0pt,minimum size = 6pt}]
\node[X] (x) at (-1,0)  {} ;
\node[X] (x') at (0,-1)  {};
\node[Y] (y) at (0,1)  {};
\node[Y] (y') at (1,0)  {};
\draw[-,thick, dashed] (-1,0)--(0,0)--(0,1);
\draw[-,thick] (0,-1)--(0,0)--(1,0);
\node [below left=of x] {$x$};
\node [below left=of x'] {$x'$};
\node [above right=of y] {$y$};
\node [above right=of y'] {$y'$};
\end{tikzpicture}
\end{minipage}
\begin{minipage}{0.32\linewidth}
\centering
\begin{tikzpicture}
[node distance = 0pt, X/.style={circle,draw=purple,fill=purple,thick,inner sep = 0pt,minimum size = 6pt},
Y/.style={rectangle,draw=NavyBlue,fill=NavyBlue,thick,inner sep = 0pt,minimum size = 6pt}]
\node[X] (x) at (-1,0)  {} ;
\node[X] (x') at (0,-1)  {};
\node[Y] (y) at (0,1)  {};
\node[Y] (y') at (1,0)  {};
\draw[-,thick, dashed] (-1,0)--(-0.5,0)--(0,0.5)--(0,1);
\draw[-,thick] (0,-1)--(0,-0.5)--(0.5,0)--(1,0);
\node [below left=of x] {$x$};
\node [below left=of x'] {$x'$};
\node [above right=of y] {$y$};
\node [above right=of y'] {$y'$};
\end{tikzpicture}
\end{minipage}
\begin{minipage}{0.32\linewidth}
\centering
\begin{tikzpicture}
[node distance = 0pt, X/.style={circle,draw=purple,fill=purple,thick,inner sep = 0pt,minimum size = 6pt},
Y/.style={rectangle,draw=NavyBlue,fill=NavyBlue,thick,inner sep = 0pt,minimum size = 6pt}]
\node[X] (x) at (-1,0)  {} ;
\node[X] (x') at (0,-1)  {};
\node[Y] (y) at (0,1)  {};
\node[Y] (y') at (1,0)  {};
\draw[-,thick, dashed] (-1,0)--(0,1);
\draw[-,thick] (0,-1)--(1,0);
\node [below left=of x] {$x$};
\node [below left=of x'] {$x'$};
\node [above right=of y] {$y$};
\node [above right=of y'] {$y'$};
\end{tikzpicture}
\end{minipage}
\caption{Deviant geodesics in $\dgmlpq$ for $p=q=1$.}
\label{fig:deviant11}
\end{figure}

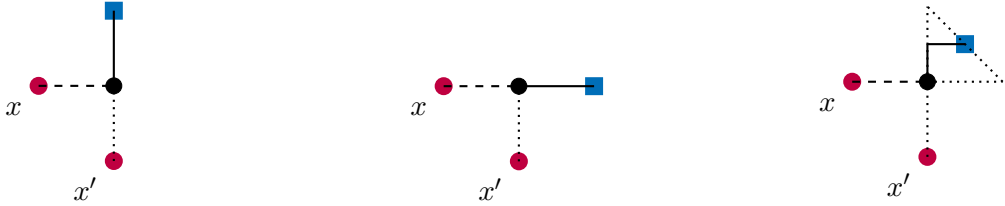
\begin{figure}
\begin{minipage}{0.32\linewidth}
\centering
\begin{tikzpicture}
[node distance = 0pt, X/.style={circle,draw=purple,fill=purple,thick,inner sep = 0pt,minimum size = 6pt},
Y/.style={rectangle,draw=NavyBlue,fill=NavyBlue,thick,inner sep = 0pt,minimum size = 6pt}]
\node[X] (x) at (-1,0)  {} ;
\node[X] (x') at (0,-1)  {};
\node[Y] (y) at (0,1)  {};
\node[] (y') at (1,0)  {};
\node[circle,draw,fill,inner sep = 0pt,minimum size = 6pt] (y') at (0,0)  {};
\draw[-,thick, dashed] (-1,0)--(0,0);
\draw[-,thick,dotted] (0,-1)--(0,0);
\draw[-,thick] (0,0)--(0,1);
\node [below left=of x] {$x$};
\node [below left=of x'] {$x'$};
\end{tikzpicture}
\end{minipage}
\begin{minipage}{0.32\linewidth}
\centering
\begin{tikzpicture}
[node distance = 0pt, X/.style={circle,draw=purple,fill=purple,thick,inner sep = 0pt,minimum size = 6pt},
Y/.style={rectangle,draw=NavyBlue,fill=NavyBlue,thick,inner sep = 0pt,minimum size = 6pt}]
\node[X] (x) at (-1,0)  {} ;
\node[X] (x') at (0,-1)  {};
\node[] (y) at (0,1)  {};
\node[Y] (y') at (1,0)  {};
\node[circle,draw,fill,inner sep = 0pt,minimum size = 6pt] (y') at (0,0)  {};
\draw[-,thick, dashed] (-1,0)--(0,0);
\draw[-,thick,dotted] (0,-1)--(0,0);
\draw[-,thick] (0,0)--(1,0);
\node [below left=of x] {$x$};
\node [below left=of x'] {$x'$};
\end{tikzpicture}
\end{minipage}
\begin{minipage}{0.32\linewidth}
\centering
\begin{tikzpicture}
[node distance = 0pt, X/.style={circle,draw=purple,fill=purple,thick,inner sep = 0pt,minimum size = 6pt},
Y/.style={rectangle,draw=NavyBlue,fill=NavyBlue,thick,inner sep = 0pt,minimum size = 6pt}]
\node[X] (x) at (-1,0)  {} ;
\node[X] (x') at (0,-1)  {};
\node[Y] (y) at (0.5,0.5)  {};
\node[circle,draw,fill,inner sep = 0pt,minimum size = 6pt] (y') at (0,0)  {};
\draw[-,thick, dashed] (-1,0)--(0,0);
\draw[-,thick,dotted] (0,-1)--(0,0);
\draw[-,thick,dotted] (0,0)--(1,0);
\draw[-,thick,dotted] (0,0)--(0,1);
\draw[-,thick,dotted] (1,0)--(0,1);
\draw[-,thick] (0,0.5)--(0.5,0.5);
\draw[-,thick] (0,0)--(0,0.5);
\node [below left=of x] {$x$};
\node [below left=of x'] {$x'$};
\end{tikzpicture}
\end{minipage}\caption{Branching geodesics in $\dgmlpq$ for $p=q=1$. The rightmost figure depicts points $x$ and $x'$ which move at the same speed to the center, merge, and travel up and to the right along the solid line. }
\label{fig:branch11}
\end{figure}

\section{Characterization of geodesics}
\label{sec:geod-char}

\subsection{A preliminary result about limiting bijections}

We now collect a lemma (Lemma \ref{lem:limiting-bijection}) showing how, given a sequence of bijections between persistence diagrams, we can pick out a subsequence of bijections that converges pointwise to a limiting bijection. This lemma is used directly in proving Theorem \ref{thm:char-II} from Theorem \ref{thm:char-I}, and as a corollary we also obtain the existence of optimal bijections between diagrams, which is used throughout the proof of Theorem \ref{thm:char-I}. The main proof technique is a standard diagonal argument with some additional consideration for the multiset nature of persistence diagrams, and a similar proof appeared in \cite[Proposition 1]{turner2013means}. 

Our reason for viewing persistence diagrams in $\R^2 \times \N$ becomes apparent in this section. We view $\R^2 \times \N$ as a subset of $\R^3$ endowed with the subspace topology. This allows us to invoke the Bolzano-Weierstrass theorem to obtain convergent sequences. 


\newcommand{\yy}{\mathbf{y'}}

\newcommand{\xx}{\mathbf{x'}}
\newcommand{\nn}{^{(n)}}
\newcommand{\ii}{^{(i)}}
\newcommand{\jj}{^{(j)}}
\newcommand{\Psixy}{\Psi^{XY}}
\newcommand{\Psiyx}{\Psi^{YX}}

\medskip

\paragraph{Notation} Below we will write $\lim$ to denote limits with respect to the usual topology in $\R^2$ or $\R^3$. This is different from a $\dlpq$ limit, which uses only the geometric component of a point in a persistence diagram and ignores the indexing coordinate. Also, we interchangeably write $X_>$ or $X\setminus \Delta$ to denote the off-diagonal points of a persistence diagram $X$, depending on which notation better preserves typography. To emphasize that each point in a persistence diagram is a vector, we use boldface notation, e.g. $\x$ or $\y$.

For a point $a$ in a persistence diagram, we write $\pi_{\Delta}(a) \in \R^2$ to denote its projection onto the diagonal, ignoring the indexing coordinate $\N$. In other words, it is the shorthand notation for projecting a point to its geometric component in $\R^2$, and then further projecting the resulting point to the diagonal.

\newcounter{saveenum}
\begin{lemma} 
\label{lem:bw-r3}
Fix $p,q \in [1,\infty]$. Let $X, Y \in \dgmlpq$, and let $\y \in Y\setminus \Delta$. Then, 
\begin{enumerate}
\item There exists $\e >0 $ such that $B_{\R^3}(\y,\e) \cap Y = \{\y\}$.
\setcounter{saveenum}{\value{enumi}}
\end{enumerate}
Suppose also that $\Phi_k:X \r Y$ is a sequence of bijections and $\x \in X$ is such that $\lim_k \Phi_k(\x) = \y$.
\begin{enumerate}
\setcounter{enumi}{\value{saveenum}}
\item Then there exists $k_0 \in \N$ such that for all $k \geq k_0$, we have $\Phi_k(\x) = \y$.
\setcounter{enumi}{\value{saveenum}} 
\end{enumerate}
Finally, suppose $\yy \in \Delta$ and $\xx \in X\setminus \Delta$ is such that $\lim_k \Phi_k(\xx) = \yy$. 
Suppose also that $C_p[l^q](\Phi_k) \r \dlpq(X,Y)$ as $k \r \infty$. 
Then,
\begin{enumerate}
\setcounter{enumi}{\value{saveenum}+1} 
\item $\pi_{\R^2}(\yy) = \pi_{\Delta}(\xx)$ (i.e. optimal bijections map $\xx$ to the diagonal via orthogonal projection).
\end{enumerate}

\end{lemma}
\begin{proof}[Proof of Lemma \ref{lem:bw-r3}]
Let $\e>0$ be small enough so that $B_{\R^3}(\y,2\e) \cap \Delta = \emptyset$, and define  $U:=B_{\R^3}(\y,\e) \cap Y \cap \Delta = \emptyset$. Then $U$ has a strictly positive distance to the diagonal. Since $Y \in \dgmlpq$, there can only be finitely many points, including multiplicity, in $U$. Different copies of $\y$ in $U$ have the same $\R^2$ coordinates, but differ on the $\N$ coordinate by at least 1. Thus $\e$ can be made sufficiently small so that $B_{\R^3}(\y,\e) \cap Y = \{\y\}$. This proves the first assertion. The second assertion follows immediately. 

The third assertion is also easy to see, and we provide a few lines of proof.
 Suppose toward a contradiction that $\yy \in \Delta$ and $\lim_k \Phi_k(\xx) = \yy$, but $\pi_{\R^2}(\yy) \neq \pi_{\Delta}(\xx)$, i.e. $\yy$ is not the diagonal projection of $\xx$. 
 Then there exists $\e > 0$ and $\eta_\e >0$ such that the distance from $\pi_{\R^2}(\xx)$ to $B_{\R^2}(\pi_{\R^2}(\yy),\e)$ is at least 
$\norm{\pi_{\R^2}(\xx) - \pi_{\Delta}(\xx)}_{q} + \eta_\e$. Thus there exists $k_0 \in \N$ such that for all $k \geq k_0$, $\norm{\pi_{\R^2}(\xx) - \pi_{\R^2}(\Phi_k(\xx))}_{q} > \norm{\pi_{\R^2}(\xx) - \pi_{\Delta}(\xx)}_{q} + \eta_\e$. But then $C_p[l^q](\Phi_k) \geq \dl(X,Y) + \eta_\e$ for all $k \geq k_0$. This is a contradiction. \end{proof}

Here is the main result of this section. 

\begin{lemma}[Limiting bijections]
\label{lem:limiting-bijection}
Let $p,q \in [1,\infty]$.
Let $X,Y \in \dgmlpq$, and let $\Phi_k:X \r Y$ be a sequence of bijections such that $C_p[l^q](\Phi_k) \r \dlpq(X,Y)$. Then there exists a subsequence indexed by $L \subseteq \N$ and a limiting bijection $\Phi_*$ such that $\Phi_k \xrightarrow{k \in L,\, k \r \infty} \Phi_*$ pointwise and $C_p[l^q](\Phi_*) = \dlpq(X,Y)$. 
\end{lemma}

\begin{proof}[Proof of Lemma \ref{lem:limiting-bijection}]
For each $\Phi_k$, we let $\Psixy_k:X \r Y|_{\R^2}$ denote the geometric part (i.e. the $\R^2$ component) of $\Phi_k$. 
We also write $\Psiyx_k: Y \r X|_{\R^2}$ to denote the geometric component of the inverse map $\Phi_k\inv$. 
Recall that only the geometric component is involved in $\dlpq$ computations (cf. Definition \ref{def:pdgm}).

Define $Y_0:= (Y\setminus \Delta) \cup 
\{\y \in \Delta : \pi_{\R^2}(\y) = \pi_{\Delta}(\x),\; \x \in X\setminus \Delta\}$. Then $Y_0$ denotes the union of the countably many off-diagonal points of $Y$ with the countably many copies of diagonal points that are projections of off-diagonal points in $X$. 
This is a countable set. Fix an enumeration $Y_0 = \{\y^{(n)}\}_{n \in \N}$. 

Since $X,Y \in \dgml$, $\dl(X,Y) < \infty$, and $C_p[l^q](\Phi_k) \r \dl(X,Y)$, we know $\lp \Psiyx_k(\y\ii)\rp_k$ is a bounded sequence in $\R^2$ for each $i \in \N$. By a diagonal argument and the Bolzano-Weierstrass theorem, we obtain a diagonal subsequence indexed by $J \subseteq \N$ such that $(\Psiyx_k)_{k \in J}$ converges pointwise on $Y_0$. Define $\Psiyx_*$ on $Y_0$ by setting $\Psiyx_*(\y) := \lim_{k \r \infty, \, k\in J}\Psiyx_k(\y)$ for each $\y \in Y_0$.
Note that if $\Psiyx_*(\y\ii) \in \Delta$ for some $\y\ii \in Y_0 \setminus \Delta$, then by an argument analogous to that of Lemma \ref{lem:bw-r3}, we have $\Psiyx_*(\y\ii) = \pi_{\Delta}(\y\ii) \in X$.

Next define:
\begin{align*}
Q:= \{( \pi_{\Delta}(\y\ii) , i ) : \y\ii \in Y_0 \setminus \Delta, \; \Psiyx_*(\y\ii) \in \Delta \},  && X_1:= (X\setminus \Delta) \sqcup 
Q.
\end{align*}

$Q$ contains all the diagonals of $X$ matched to off-diagonals in $Y$. $X_1$ is countable, so another application of a diagonal argument and the Bolzano-Weierstrass theorem gives a subsequence indexed by $K\subseteq J$ such that $(\Psixy_k)_{k \in K}$ converges pointwise on $X_1$. Define $\Psixy_*(\x):= \lim_{k \r \infty, \, k \in K}(\Psixy_k(\x))$ for each $\x \in X_1$.

Next write $Y_0^{\offd}$ and $Y_0^{\ond}$ to denote the off-diagonal and on-diagonal points of $Y_0$, respectively.
Define $A:= \{\x \in (X\setminus \Delta) : \Psixy_*(\x) \in Y_0^{\offd}|_{\R^2}\}$ and $B:= (X\setminus \Delta) \setminus A$. Fix an enumeration $\{\x\nn\}_{n \in \N}$ on $B$. Note the following descriptions of the sets $A$ and $B$ in terms of how they should be matched by the limiting bijection: $A$ contains all the off-diagonal points of $X$ that are matched to off-diagonals in $Y$,  and $B$ contains all the off-diagonals of $X$ matched to diagonals in $Y$. In particular, $X_1 = A \sqcup B \sqcup Q$.

Each point in $Y_0^{\offd}$ has finite multiplicity, because otherwise we would have $Y \not\in\dgml$. Thus for any $\x \in A$,  $\lp \Phi_{k}(\x) \rp_{k\in K}$ is a bounded sequence in $\R^2 \times \N$ by Lemma \ref{lem:bw-r3}. By the Bolzano-Weierstrass theorem and a diagonal argument as above, we get a subsequence $\lp \Phi_{k}\rp_{k\in L}$ indexed by $L\subseteq K$ converging pointwise on $A$. Since $L\subseteq K$, we have $\lim_{k\r \infty,\, k\in L} \pi_{\R^2}\lp\Phi_{k}(\x)\rp = \Psixy_*(\x)$ for each $\x \in A$. 

Define $\Phi_*:X_1 \r Y$ by writing the following for each $\x \in X_1$:
\[\Phi_* (\x):= 
\begin{cases}
\lim_{k \r \infty,\, k \in L} \Phi_{k}(\x) &: \x \in A\\
\y\ii &: \x \in Q,\, \x = (\pi_{\Delta}(\y\ii),i) \\
(\pi_{\Delta}(\x\ii),i) &: \x \in B,\, \x = \x\ii. 
\end{cases}
\]

\noindent
\textbf{Claim: $\Phi_*|_A: A \r Y_0^{\offd}$ is injective.} Let $\x,\xx \in A$ be such that $\Phi_*(\x) = \Phi_*(\xx)$. Write $\y:= \Phi_*(\x)$. 
By Lemma \ref{lem:bw-r3}, we obtain $\e> 0$ and $k_0 \in \N$ such that $\Phi_{k}(\x), \Phi_{k}(\xx) \in B_{\R^3}(\y,\e)$ for all $ k \geq k_0$, $k \in L$. Thus for such $k$, $\x = \Phi_{k}^{-1}(\y) = \xx$.

\noindent
\textbf{Claim: $\Phi_*|_Q: Q \r Y_0^{\offd}$ is injective.}
Let $\x,\xx \in Q$ be such that $\Phi_*(\x) = \Phi_*(\xx) = \y\ii$. Then $\x = (\pi_{\Delta}(\y\ii),i) = \xx$ by the definition of $\Phi_*$ on $Q$.

\noindent
\textbf{Claim: $\Phi_*|_{A\cup Q}: A\cup Q \r Y_0^{\offd}$ is injective.}
We have already dealt with the cases $\x,\xx \in A$ and $\x,\xx \in Q$. Now we deal with the remaining case.
Let $\x \in A$, $\xx \in Q$ be such that $\Phi_*(\x) = \y\ii = \Phi_*(\xx)$ for some $\y\ii \in Y\setminus \Delta$. By Lemma \ref{lem:bw-r3}, there exists $k_0$ such that for all $k\geq k_0$, $\Phi_k(\x) = \y\ii$. Then for all such $k$, $\Psi^{YX}_k(\y\ii) = \pi_{\R^2}(\x)$, which is bounded away from $\Delta$. On the other hand, since $\Phi_*(\xx) = \y\ii$, we know that $\lim_{k \r \infty, \, k \in L}\Psi^{YX}_k(\y\ii) = \pi_{\R^2}(\xx) = \pi_{\Delta}(\y\ii) \in \Delta$. This is a contradiction.

\noindent
\textbf{Claim: $\Phi_*|_{A\cup Q}: A \cup Q \r Y_0^{\offd}$ is surjective.}
Let $\y\ii \in Y_0^{\offd}$, and consider $\Psiyx_*(\y\ii)$. 
There are two cases: $\Psiyx_*(\y\ii)$ is either off-diagonal or on-diagonal. Suppose first that $\Psiyx_*(\y\ii)$ is off-diagonal. 
Then by an argument similar to that of Lemma \ref{lem:bw-r3}, we obtain $k_0 \in \N$ such that for all $k \geq k_0$, $k \in L$, $\Psiyx_k(\y\ii) = \Psiyx_*(\y\ii)$. Since $X \in \dgml$, there are only finitely many $\x \in X$ such that $\pi_{\R^2}(\x) = \Psiyx_*(\y\ii)$. Let $X_{\y\ii}:= \{\x_1,\x_2,\ldots, \x_n\}$ denote this collection. 

We know that $\Phi\inv_k(\y\ii) \in X_{\y\ii}$ for all $k \geq k_0$, $k \in L$. By the pigeonhole principle, choose a subsequence indexed by $M \subseteq L$ such that $(\Phi\inv_k(\y\ii))_{k \in M,\, k \geq k_0}$ is constant. Let $\x \in X_{\y\ii}$ denote the value of this constant sequence. Then for all $k \geq k_0,\, k \in M$, we have $\Phi_k(\x) = \y\ii$. Since $\x \in A$ and $(\Phi_k)_{k \in L}$ converges pointwise on $A$, we know furthermore that $\Phi_*(\x) = \y\ii$.

Suppose next that $\Psiyx_*(\y\ii)$ is on-diagonal. Then by what we have observed before, $\Psiyx_*(\y\ii) = \pi_{\Delta}(\y\ii) \in \Delta$. By definition, $Q$ contains $(\pi_{\Delta}(\y\ii), i)$, and $\Phi_*$ maps this to $\y\ii$.

\medskip
\noindent
\textbf{Claim: $\Phi_*|_B: B \r Y_0^{\ond}$ is injective.}

First note that the codomain of $\Phi_*|_B$ is not $Y_0^{\ond}$ a priori, because the points in $\im(\Phi_*|_B)$ and $Y_0^{\ond}$ may differ on the $\N$ coordinate. But this is simply a matter of choosing representatives from the $\N$-indexed diagonal points, and we may relabel the $\N$-coordinates of points in $Y_0^\ond$ to have $\im(\Phi_*|_B) \subseteq Y_0^{\ond}$.

To see injectivity, suppose $\x\ii, \x\jj \in B$ are such that $\Phi_*(\x\ii) = \Phi_*(\x\jj)$. Then $(\Psi_*(\x\ii),i) = (\Psi_*(\x\jj),j)$, so $i = j$ and hence $\x\ii = \x\jj$.

Finally we extend $\Phi_*$ to a bijection from $X$ to $Y$ by matching the points of $X\setminus X_1$ to the points of $Y\setminus Y_0$ and $Y_0^{\ond} \setminus \im(\Phi_*|_B)$, all of which are diagonal. We continue writing $\Phi_*:X \r Y$ to denote this bijection. It follows from the construction that $\Phi_*$ satisfies the statement of the theorem. This concludes the proof. \end{proof}

As a corollary of this lemma, we see that $\dgmlpq$ is a geodesic space. This result was already implicit in \cite[Proposition 1]{turner2013means}, where it was stated in the case $p = q \in [1,\infty]$. See also \cite{tmmh} for a different argument in the case $p=q=2$.  
In addition to using the result about existence of geodesics throughout this paper, we specifically use Lemma \ref{lem:limiting-bijection} to prove Theorem \ref{thm:char-II} via Theorem \ref{thm:char-I}.

\begin{corollary} 
\label{cor:geod-exist}
Fix $p,q \in [1,\infty]$. Let $X,Y \in \dgmlpq$. Then there exists a bijection $\Phi:X \r Y$ such that $C_p[l^q](\Phi) = \dlpq(X,Y)$. Thus we immediately have a convex-combination geodesic from $X$ to $Y$. 
\end{corollary}

\begin{proof}[Proof of Corollary \ref{cor:geod-exist}] 
Lemma \ref{lem:limiting-bijection} yields an optimal bijection $\Phi$. Let $\g$ denote the associated convex combination curve. To conclude, we need to show that $\g$ is geodesic. Let $s,t \in [0,1]$. To compare $\g(s)$ and $\g(t)$, consider the bijection associating $(1-t)x+ t\Phi(x)$ with $(1-s)x + s\Phi(s)$. Then we have:
\begin{align*}
d_p[l^q](\g(s),\g(t)) &\leq \lp \sum_{x \in X} \lnorm (1-t)x+ t\Phi(x) - (1-s)x- s\Phi(x) \rnorm_q^p \rp^{1/p} \\
&= \lp \sum_{x \in X} \lnorm (s-t) (x - \Phi(x) ) \rnorm_q^p \rp^{1/p} \\
&= |t-s| \lp \sum_{x\in X} \lnorm x - \Phi(x) \rnorm_q^p \rp^{1/p} = |t-s| d_p[l^q](X,Y).
\end{align*}
By a property of geodesics, showing the inequality is sufficient to guarantee equality (cf. \cite[Lemma 1.3]{dgh-era}). This concludes the proof. \end{proof}

\subsection{Lemmas related to the characterization of geodesics}

{

The proof of Theorem \ref{thm:char-I} will follow the strategy used by Sturm in proving an analogous result about geodesics in the space of metric measure spaces \cite{sturm2012space}. We first present a sequence of lemmas.

\begin{lemma}[Clarkson's inequality, see \cite{clarkson1936uniformly, boas1940some}] 
\label{lem:clarkson}
Let $p \in [2,\infty)$, and let $v,w \in l^p$. Then,
\begin{align*} 
\norm{ v + w}_p^p + \norm{v-w}^p_p & \leq 2^{p-1}\lp \norm{v}_p^p + \norm{w}_p^p\rp.
\end{align*}
\end{lemma}

\begin{lemma}[Application of Lemma \ref{lem:clarkson}]
\label{ineq:sturm-2}
Let $t\in (0,1)$, $p \in [2,\infty)$. Then there exists a constant $C>0$ depending on $p$ and $t$ such that for any $v, w \in l^p$, we have
\[ \norm{tv + (1-t)w}_p^p \leq t\norm{v}_p^p +(1-t)\norm{w}_p^p - t(1-t)C\norm{v-w}_p^p.\]
\end{lemma}

\begin{lemma}[BCL inequality, see \cite{ball1994sharp} Proposition 3]
\label{lem:bcl}
Let $p \in (1,2]$, and let $v,w \in l^p$. Then, 
\[\norm{ v + w}_p^2 + \norm{v-w}_p^2 \geq 2\norm{v}_p^2 + 2(p-1)\norm{w}_p^2.\]
\end{lemma}

\begin{lemma}[Application of BCL inequality]
\label{ineq:sturm-3}
Let $t\in (0,1)$, $p \in (1,2]$. Then for any vectors $v, w \in l^p$, we have
\[ \norm{tv + (1-t)w}_p^2 \leq  t\norm{v}_p^2 +(1-t)\norm{w}_p^2 - (p-1)t(1-t)\norm{v-w}_p^2.\]
\end{lemma}

\begin{lemma}[Application of Jensen's inequality]
\label{ineq:sturm-1}
Let $a_1,\ldots, a_n \in [0,\infty)$, $t_0 < t_1 < \ldots < t_n \in \R$, and let $p \in (1,\infty)$. Then,
\[\frac{1}{(t_n - t_0)^{p-1}} \lp \sum_{i=1}^na_i \rp^p \leq \sum_{i=1}^n\frac{a_i^p}{(t_i-t_{i-1})^{p-1}}.\]
\end{lemma}


\newcommand{\tb}{\tfrac{2a}{2^{b+1}}}
\newcommand{\tbo}{\tfrac{2a+1}{2^{b+1}}}
\newcommand{\tbt}{\tfrac{2a+2}{2^{b+1}}}
\newcommand{\cb}{\tfrac{1}{C(b,p)}}
\newcommand{\cbo}{\tfrac{1}{C(b+1,p)}}
\newcommand{\ft}{\tfrac{v}{2}}
\newcommand{\gt}{\tfrac{w}{2}}
\newcommand{\ont}{\tfrac{1}{2}}
\newcommand{\ontbo}{\tfrac{1}{2^{b+1}}}

\begin{proof}[Proof of Lemma \ref{ineq:sturm-1}]

For each $1\leq i \leq n$, write $\lambda_i = \tfrac{t_i - t_{i-1}}{t_n - t_0}$ and $x_i = \tfrac{a_i}{t_i-t_{i-1}}$. Notice that $\sum_{i=1}^n \lambda_i =\sum_{i =1}^n\tfrac{t_i - t_{i-1}}{t_n - t_0} = 1$.

By Jensen's inequality, 
\begin{align*}
\lp \sum_{i=1}^n \lambda_ix_i\rp^p 
&= \lp \sum_{i=1}^n\frac{a_i}{t_i - t_{i-1}}\cdot \frac{t_i-t_{i-1}}{t_n-t_0} \rp^p
= \frac{1}{(t_n-t_0)^p}\lp \sum_{i=1}^na_i \rp^p \\
& \leq 
 \sum_{i=1}^n \lp \frac{a_i}{t_i - t_{i-1}}\rp^p \cdot \frac{t_i - t_{i-1}}{t_n - t_0}
= \frac{1}{t_n - t_0}\sum_{i=1}^n\frac{a_i^p}{(t_i - t_{i-1})^{p-1}}.
\end{align*}

This verifies Lemma (\ref{ineq:sturm-1}). \end{proof}

\begin{proof}[Proof of Lemma \ref{ineq:sturm-2}]
It suffices to show the inequality for dyadic rationals in the unit interval, and then invoke the density of the dyadic rationals. We consider dyadic rationals of the form $t=a/2^b$, for integers $b\geq 0$ and $0\leq a \leq 2^b$. The proof is by induction, based on the following inductive hypothesis: for each $b \in \N$, there exists a constant $C>0$ depending on $b$ and $p$ such that:
\[\lnorm  (\tfrac{a}{2^{b}})v + (1-\tfrac{a}{2^{b}})w \rnorm_p^p \leq (1- \tfrac{a}{2^{b}})\lnorm w \rnorm^p_p + (\tfrac{a}{2^{b}})\lnorm v \rnorm^p_p - C(1- \tfrac{a}{2^{b}})(\tfrac{a}{2^{b}})\lnorm w-v \rnorm^p_p.\]

For the base case, we use Lemma \ref{lem:clarkson} (a sharper result is obtained via the parallelogram law for $p=2$):

\begin{align*}
\norm{\ft + \gt}^p_p \leq  2^{p-1}\norm{\ft}^p_p + 2^{p-1}\norm{\gt}^p_p - \norm{\ft - \gt}^p_p
&= \ont\norm{v}^p_p + \ont\norm{w}^p_p - \tfrac{1}{2^p} \norm{v - w}^p_p\\
&= \ont\norm{v}^p_p + \ont\norm{w}^p_p - C \cdot \ont \cdot \ont \norm{v - w}^p_p.
\end{align*}

Suppose the inductive hypothesis is true up to some $b \in \N$ and all $0\leq a \leq 2^b$. This means that at the $(b+1)$th step, the inductive hypothesis is true for all $2a/2^{b+1}$, where $0 \leq a \leq 2^b$. Fix $a$ in this range, and consider the midpoint $(2a+1)/2^{b+1}$ of $2a/2^{b+1}$ and $(2a+2)/2^{b+1}$. For the inductive step, we have:
\begin{align*}
\norm{ \tbo v + \lp 1 - \tbo \rp w }^p_p 
&= \norm{ \ont \lp \tb v + \tbt v \rp + \ont \lp \lp 1 - \tb\rp w + \lp 1 - \tbt \rp w \rp }^p_p\\
&= \norm{ \ont \lp \tb v + \lp 1-\tb \rp w \rp + \ont \lp \tbt v + \lp 1 - \tbt \rp w \rp  }_p^p   \\
&\leq 2^{p-1}\norm{ \ont \lp \tb v + \lp 1 - \tb \rp w \rp }^p_p + 2^{p-1}\norm{ \ont \lp \tbt v + \lp 1 - \tbt \rp w \rp }^p_p\\
& \phantom{\hspace{0.5 in}} - \norm{ \ontbo w - \ontbo v }^p_p.
\end{align*}

Here we used Lemma \ref{lem:clarkson} for the inequality, and showed the final term after simplification. The inductive hypothesis can now be applied to the first two terms. Either by hand or a computer algebra package, we see that the inductive step holds for $(b+1)$. This completes the proof of Lemma \ref{ineq:sturm-2}. \qedhere

\end{proof}

\begin{proof}[Proof of Lemma \ref{ineq:sturm-3}]
We again proceed by showing the inequality for dyadic rationals. Writing $t = a/2^b$ as before, the inductive hypothesis now becomes:
\[\lnorm  (\tfrac{a}{2^{b}})v + (1-\tfrac{a}{2^{b}})w \rnorm_p^2 \leq (1- \tfrac{a}{2^{b}})\lnorm w \rnorm^2_p + (\tfrac{a}{2^{b}})\lnorm v \rnorm^2_p - (p-1)(1- \tfrac{a}{2^{b}})(\tfrac{a}{2^{b}})\lnorm w-v \rnorm^2_p.\]

For the base case, we use Lemma \ref{lem:bcl}. Set $c = (v+w)/2$ and $d = (v-w)/2$. Then $c + d =v$ and $c-d = w$, and we have: 

\begin{align*}
\norm{c+d}^2_p + \norm{c-d}^2_p &\geq 2\norm{c}^2_p + 2(p-1)\norm{d}^2_p, \text{ so }\\
\norm{v}^2_p + \norm{w}^2_p &\geq 2\norm{\ft + \gt}^2_p + 2(p-1)\norm{\ft - \gt}^2_p, \text{ and hence}\\
\norm{\ft + \gt}^2_p &\leq \ont\norm{v}^2_p + \ont\norm{w}^2_p - (p-1)(\ont)^2\norm{v-w}^2_p.
\end{align*}

The inductive step then applies as in the proof of Lemma \ref{ineq:sturm-2}.\qedhere

\end{proof}

}

\subsection{Geometric intuition for the proof of Theorem \ref{thm:char-I}}
\label{sec:geom-idea}

The proof of Theorem \ref{thm:char-I} proceeds via a geometric argument about $p$-norms. We abstract away these arguments and present the intuition here.

The setup of Theorem \ref{thm:char-I} involves two persistence diagrams $X$ and $Y$ with finitely many off-diagonal points and a geodesic $\mu:[0,1] \r \dgmlpq$ from $X$ to $Y$. Let $\Phi$ be a bijection between $X$ and $Y$ (we will specify this bijection in the actual proof), and let $\gamma$ be the convex-combination curve in $\dgmlpq$ from $X$ to $Y$ induced by $\Phi$. For each $t \in [0,1]$, let $x_t := (1-t)x + t \Phi(x) \in \gamma(t)$. Fix $t \in (0,1)$. Suppose also that we have a bijection between $\gamma(t)$ and $\mu(t)$ (this will be specified in the proof). Let $\psi(x_t) \in \mu(t)$ denote the image of $x_t$ under this bijection.

In what follows, we will write norms in some $l^q$ space raised to some power $p>1$ (i.e. terms of the form $\norm{\cdot}_q^p$) without specifying the elements of the normed space, but this will be written explicitly in the proof. Our goal is to prove that 
\[ \sum_{x \in X} \lnorm x_t - \psi(x_t) \rnorm_q^p = 0,\] 
which would show that $\dlpq(\gamma(t),\mu(t)) = 0$. To approach this, consider the following quantities:
\[ \lnorm x - \psi(x_t)\rnorm_q^p, \qquad \lnorm \psi(x_t) - \Phi(x) \rnorm_q^p, \qquad \lnorm x - x_t \rnorm_q^p, \qquad \lnorm x_t - \Phi(x) \rnorm_q^p.\]

\begin{figure}
\centering
\begin{tikzpicture}
[node distance = 0pt, X/.style={circle,draw=purple,fill=purple,thick,inner sep = 0pt,minimum size = 6pt},
Y/.style={rectangle,draw=NavyBlue,fill=NavyBlue,thick,inner sep = 0pt,minimum size = 6pt}]
\node[X] (x) at (0,0)  {} ;
\node[X] (phix) at (10,0)  {};
\node[X] (psix) at (4,2)  {};
\node[X] (xt) at (2.5,0)  {};


\draw[->,thick] (phix)--(x);
\draw[->,thick] (phix)--(psix) node[pos=0.65, above right,rotate = -20]{$\psi(x_t) - \Phi(x)$} ;
\draw[->,thick] (psix)--(x) node[midway, above, rotate=30]{$x - \psi(x_t)$};
\draw[->,thick] (psix)--(xt);
\node [below =of x] {$x$};
\node [below left=of phix] {$\Phi(x)$};
\node [below right=of xt] {$x_t = (1-t)x + t\Phi(x)$};
\node [above right=of psix] {$\psi(x_t)$};

\end{tikzpicture}
\caption{Geometric setup for the proof of Theorem \ref{thm:char-I}.}
\label{fig:geom-intuit}
\end{figure}
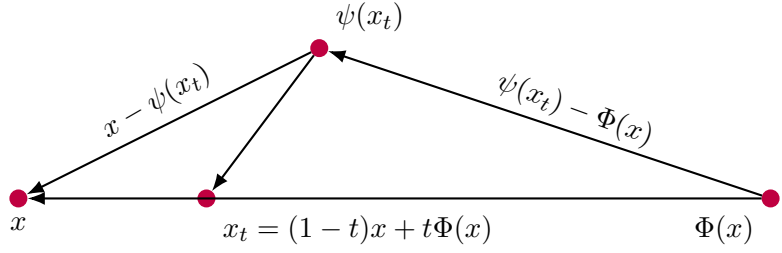

These are of course related geometrically, as suggested by Figure \ref{fig:geom-intuit}. More specifically, write $Q(x):= (x - \psi(x_t))/t$ and $R(x):= (\psi(x_t) - \Phi(x))/(1-t)$. Then the following is true:
\begin{align*}
\lnorm Q(x) - R(x) \rnorm_q^p &= \lnorm \frac{ (1-t)(x - \psi(x_t)) - t(\psi(x_t) - \Phi(x))}{t(1-t)} \rnorm_q^p\\
&= \frac{1}{t^p(1-t)^p} \lnorm (1-t)x + t\Phi(x) - \psi(x_t) \rnorm_q^p \\
&=\frac{1}{t^p(1-t)^p} \lnorm x_t - \psi(x_t) \rnorm_q^p.
\end{align*}

So to show $\sum_{x \in X} \lnorm x_t - \psi(x_t) \rnorm_q^p = 0$, it suffices to show 
\[ \sum_{x \in X} \lnorm Q(x) - R(x) \rnorm_q^p = 0. \]

To obtain $Q(x) - R(x)$, one computes:
\begin{align*}
\dlpq(X,Y)^p &\leq \sum_{x\in X} \lnorm x - \Phi(x) \rnorm_q^p \\
&= \sum_{x \in X} \lnorm x - \psi(x_t) + \psi(x_t) - \Phi(x) \rnorm_q^p\\
&= \sum_{x \in X} \lnorm tQ(x) + (1-t)R(x) \rnorm_q^p.
\end{align*}

Suppose also that one can bound:
\begin{equation}
\lnorm tQ(x) + (1-t)R(x) \rnorm_q^p \leq t \lnorm Q(x) \rnorm_q^p + (1-t) \lnorm R(x) \rnorm_q^p - C\lnorm Q(x) - R(x) \rnorm_q^p
\label{eq:geom-1}
\end{equation}
for some positive constant $C$. If one can show that 
\begin{align*}
\sum_{x\in X} t \lnorm Q(x) \rnorm_q^p + (1-t) \lnorm R(x) \rnorm_q^p = \dlpq(X,Y)^p,
\end{align*}
then by summing over $x\in X$ in Inequality (\ref{eq:geom-1}), we necessarily have $\sum_{x \in X} \lnorm Q(x) - R(x) \rnorm_q^p = 0$, which is what we need.

Notice that if $\psi(x_t) = x_t$ for each $x \in X$, i.e. if $\gamma(t) = \mu(t)$, then
\[\lnorm x - \psi(x_t) \rnorm_q = t\lnorm x - \Phi(x) \rnorm_q \text{ and }  
\lnorm \psi(x_t) - \Phi(x) \rnorm_q = (1-t)\lnorm x - \Phi(x) \rnorm_q, \text{ so} \]
\[
\lnorm Q(x) \rnorm_q = \lnorm x - \Phi(x) \rnorm_q = \lnorm R(x) \rnorm_q.
\]
Thus we have 
\begin{align*}
\sum_{x\in X} t \lnorm Q(x) \rnorm_q^p + (1-t) \lnorm R(x) \rnorm_q^p  
= \sum_{x\in X} t \lnorm x - \Phi(x) \rnorm_q^p + (1-t) \lnorm x - \Phi(x) \rnorm_q^p 
= \dlpq(X,Y)^p.
\end{align*} 

Now we proceed to the main result.

\subsection{Proof of the characterization result}

\begin{proof}[Proof of Theorem \ref{thm:char-I}]

We split the proof into two parts: first we construct bijections $\Phi_{k}$ that induce geodesics $\g_{k}$ which agree with $\mu$ at all $i2^{-k}$, for integers $0 < i < 2^k$. Then we will use the sequence $\lp \Phi_{k} \rp_k$ to construct a ``limiting bijection" that induces a geodesic satisfying the statement of the theorem.

We also alert the reader to certain notational choices we will make in this proof. We will occasionally deal with infinite-dimensional vectors $V \in \R^\N$. Recall that when there is a function $f$ defined on each element of $V$, we will write $f(V)$ to denote $(f(v_1),f(v_2),\ldots )$. Whenever we use $\norm{\cdot}_{l^p}$ notation, we assert that the vector in the argument does indeed belong to $l^p$. Typically this will be easy to see, and we will remind the reader to this effect.

\medskip
\paragraph*{Part I}

For this proof, we will use an argument about dyadic rationals. 
Fix $k \in \N$. For each $0\leq i \leq 2^k-1$, fix optimal bijections $\ph_i: \mu(i2^{-k}) \r \mu((i+1)2^{-k})$ (Corollary \ref{cor:geod-exist}). Composing these bijections together gives a bijection $\Phi_k: X \r Y$. Let $\g_k$ be the convex-combination curve induced by $\Phi_k$. Also for each  $0\leq i \leq 2^k-1$, let $\psi_i:\g_k(i2^{-k}) \r \mu(i2^{-k})$ denote the bijection induced by the $\Phi_k$ and $\ph_i$ terms. There is a choice here: one can pass to $X$ and then use the maps $\ph_0, \ph_1, \ph_2,\ldots$ or pass to $Y$ and then use the inverses of the $\ph_i$ maps. This choice will not matter, so for convenience, suppose we make the former choice. 

We wish to show $\dlpq(\g_k(i2^{-k}), \mu(i2^{-k})) = 0$ for each $i$. For notational convenience, define $t_i := i2^{-k}$ for each $0\leq i \leq 2^k$. Also for each $x\in X$ and each $t \in [0,1]$, define $x_t := (1-t)x + t \Phi_k(x)$. Note that each $x_{t_i}$ belongs to $\g_k(t_i)$. 

Let $i \in \{1, \ldots, 2^k-1\}$. For each $x \in X$, define $Q(x):=\frac{x - \psi_i(x_{t_i})}{t_i - 0}$ and $R(x):= \frac{\psi_i(x_{t_i}) - \Phi_k(x)}{1 - t_i}$ (recall the idea described in \S\ref{sec:geom-idea}). 

\medskip

\paragraph{The case $p=q \in [2,\infty)$.}
First we consider the case $p=q \in [2,\infty)$.
Then we have:
\begin{align}
 \dlpp(X,Y)^p &\leq \sum_{x\in X} \norm{ x - \Phi_k(x)}_p^p \nonumber \\
&=  \sum_{x\in X} \norm{x - \psi_i(x_{t_i}) + \psi_i(x_{t_i}) - \Phi_k(x)}_p^p \nonumber \\
&= \sum_{x\in X} \norm{t_iQ(x) + (1-t_i)R(x)}_p^p \label{eq:QR} \\
&\leq \sum_{x\in X}\left[ t_i\norm{Q(x)}_p^p + (1-t_i)\norm{R(x)}_p^p - Ct_i(1-t_i)\norm{Q(x) - R(x)}_p^p \right] \label{eq:pp-clarkson}
\end{align}
For the last inequality, we have used Lemma \ref{ineq:sturm-2}. 
Specifically, $Q(x)$ and $R(x)$ are both vectors in $\C$, and we regard them as elements in $l^p$ when applying Lemma \ref{ineq:sturm-2}. 
Notice also that this is the step described in Inequality (\ref{eq:geom-1}). We split the last term into a positive part $P(p=q\in [2,\infty)):= \sum_{x\in X} \left[t_i\norm{Q(x)}_p^p + (1-t_i)\norm{R(x)}_p^p\right]$ and a negative part $N(p=q\in [2,\infty)):= \sum_{x\in X} t_i(1-t_i) \norm{Q(x) - R(x)}_p^p$ (the constant $C>0$ will not be important in the sequel). As described in \S\ref{sec:geom-idea}, our goal would be show that $N(p=q\in [2,\infty)) = 0$. 

\medskip

\paragraph{The case $q = 2, \, p \in [2,\infty)$.}
We now consider the case $q = 2, \, p \in [2,\infty)$. 
Recall the construction of $X^*$ and $Y^*$ in \S\ref{sec:OT}, as well as the functional $J_p$ defined in Equation (\ref{eq:dgm-norm}). 
By the observation in Equation (\ref{eq:bijection-approx}), we can approximate $\Phi_k$ arbitrarily well by a bijection $\s: X^* \r Y^*$ that agrees with $\Phi_k$ on off-diagonal points of both $X$ and $Y$. 
Recall also from \S\ref{sec:OT} that we regard $X^* = (x^1,x^2,x^3,\ldots )$ and $\s(X^*) = (\s(x^1),\s(x^2),\s(x^3),\ldots )$ as infinite-dimensional vectors.

Let $\e > 0$,
and let $\s:X^* \r Y^*$ be a bijection in $\Lambda_{\Phi_k}$ such that $|C_p(\Phi_k)^p - J_p(\s)^p| < \e$.  
By this choice we know $X^* - \s(X^*) \in l^p$.
Combining these observations, we have:
\begin{align*}
\dlp(X,Y)^p &\leq  J_p(\s)^p + \e =  \lnorm X^* - \s(X^*) \rnorm_{l^p}^p + \e.
\end{align*}

Now let $X^*_{t_i}$ denote the vector $(x^1_{t_i}, x^2_{t_i},\ldots )$. 
This is just $\gamma_k(i2^{-k})^*$ with a particular ordering on the elements that is consistent with the ordering initially placed on $X^*$. 


Once again, by the observation in Equation (\ref{eq:bijection-approx}), we can approximate $\psi_i: \g_k(i2^{-k}) \r \mu(i2^{-k})$ arbitrarily well by a bijection $\rho_i: X^*_{t_i} \r \mu(i2^{-k})^*$ such that $\rho_i \in \Lambda_{\psi_i}$. 
Let $\rho_i \in \Lambda_{\psi_i}$ be such that $|C_p(\psi_i)^p - J_p(\rho_i)^p| < \e$. 
Next define
\[\widetilde{Q}(X^*):= \lp \frac{x^1 - \rho_i(x^1_{t_i})}{t_i - 0}, \frac{x^2 - \rho_i(x^2_{t_i})}{t_i - 0}, \ldots \rp, \, 
\widetilde{R}(X^*):= \lp \frac{\rho_i(x^1_{t_i}) - \s(x^1)}{1- t_i }, \frac{\rho_i(x^2_{t_i}) - \s(x^2)}{ 1- t_i }, \ldots \rp.
 \] 
 The choice of $\rho_i$ ensures that $\widetilde{Q}(X^*)$ and $\widetilde{R}(X^*)$ are both in $l^p$.
Then we have:
 \begin{align*}
\lnorm X^* - \s(X^*) \rnorm_{l^p}^p + \e &=  \lnorm X^* - \rho_i(X^*_{t_i}) + \rho_i(X^*_{t_i}) - \s(X^*) \rnorm_{l^p}^p + \e\\
&= \medlnorm t_i \widetilde{Q}(X^*) + (1-t_i)\widetilde{R}(X^*) \medrnorm_{l^p}^p + \e\\
&\leq t_i \medlnorm \widetilde{Q}(X^*)  \medrnorm_{l^p}^p
+ (1-t_i)\medlnorm \widetilde{R}(X^*) \medrnorm_{l^p}^p
- Ct_i(1-t_i) \medlnorm \widetilde{Q}(X^*) - \widetilde{R}(X^*) \medrnorm_{l^p}^p + \e.
\end{align*}
where the last inequality is obtained via Lemma \ref{ineq:sturm-2}. Notice that this application uses the full strength of Lemma \ref{ineq:sturm-2}, in the sense that is used as an inequality between norms of truly infinite-dimensional vectors, as opposed to being used as an inequality between norms in $\R^2$ (cf. Inequality (\ref{eq:pp-clarkson})). 

Next we compare $\medlnorm \widetilde{Q}(X^*)\medrnorm_{l^p}^p$ with $\sum_{x\in X}\norm{Q(x)}_2^p$. Define the bijection $\a: X \r \mu(i2^{-k})$ by $x \mapsto \psi_i(x_{t_i})$. Define another bijection $\beta: X^* \r \mu(i2^{-k})^*$ by $x^j \mapsto \rho_i(x^j_{t_i})$. By our choices of $\s$ and $\rho_i$, we know that $\a$ and $\b$ agree on off-diagonal elements of $X$ and $\mu(i2^{-k})$. Furthermore, $\a$ is the identity on diagonal points that are not matched to off-diagonal points, and $\b$ incurs a total cost bounded by a function of $\e$ from moving such points infinitesimally along the diagonal. Repeating this argument for the other terms, we conclude in particular that  
\[
\dlp(X,Y)^p \leq \sum_{x\in X}\left[ t_i\norm{Q(x)}_2^p + (1-t_i)\norm{R(x)}_2^p - Ct_i(1-t_i)\norm{Q(x) - R(x)}_2^p \right]  + f(\e),
\]
where $f(\e)$ is some positive function of $\e$ that tends to zero as $\e \r 0$. 

As before, we define 
\begin{align*}
P(p\in[2,\infty),q=2)&:= \sum_{x\in X} \left[t_i\norm{Q(x)}_2^p + (1-t_i)\norm{R(x)}_2^p\right] \\
N(p\in [2,\infty),q=2)&:=  \sum_{x\in X} t_i(1-t_i) \norm{Q(x) - R(x)}_2^p.
\end{align*}

We now show how to obtain similar quantities in the final remaining case.

\paragraph{The case $q = 2, \, p \in (1,2)$.}
Let $\e > 0$. Now let $\s \in \Lambda_{\Phi_k}$ be such that $|C_p(\Phi_k)^2 - J_p(\s)^2| < \e$, and let $\rho_i \in \Lambda_{\psi_i}$ be such that $|C_p(\psi_i)^2 - J_p(\rho_i)^2| < \e$. As before, we have:
\begin{align*}
\dlp(X,Y)^2&\leq  J_p(\s)^2 + \e\\
&= \medlnorm t_i \widetilde{Q}(X^*) + (1-t_i)\widetilde{R}(X^*) \medrnorm_{l^p}^2 + \e\\
&\leq t_i \medlnorm \widetilde{Q}(X^*)  \medrnorm_{l^p}^2 
+ (1-t_i)\medlnorm \widetilde{R}(X^*)\medrnorm_{l^p}^2
- (p-1)t_i(1-t_i) \medlnorm \widetilde{Q}(X^*) - \widetilde{R}(X^*) \medrnorm_{l^p}^2 + \e,
\end{align*}
where the last inequality holds via Lemma \ref{ineq:sturm-3}. 
In this case, the argument deviates from that of the preceding cases.
We define:
\begin{align*}
P(p\in (1,2),q=2)&:= t_i \medlnorm \widetilde{Q}(X^*)  \medrnorm_{l^p}^2 
+ (1-t_i)\medlnorm \widetilde{R}(X^*)\medrnorm_{l^p}^2  + \e\\
N(p\in (1,2),q=2)&:= \medlnorm \widetilde{Q}(X^*) - \widetilde{R}(X^*) \medrnorm_{l^p}^2.
\end{align*}

In each of these three cases, we have obtained an inequality of the form $\dlpq(X,Y)^k \leq  P - CN$, where $C>0$ is some constant. 

\begin{claim}
\label{cl:positive-part}
We claim that in each case presented above, $P \leq \dlpq(X,Y)^k$ for the appropriate $k$. In the first two cases, $k = p$, and in the third case, $k = 2$. 
\end{claim}

As explained in \S\ref{sec:geom-idea}, this shows---at least in the first case---that $N = 0$, and so $\dlpq(\mu(t_i),\g_k(t_i)) = 0$. In the second and third cases, we will obtain an additional positive term $f(\e)$ which is a function of $\e$ that tends to zero as $\e \r 0$, so we will drop this term and again obtain $N=0$. So assuming Claim \ref{cl:positive-part}, and because $i$ was arbitrary, we now have $\dlpq(\mu(i2^{-k}),\g_k(i2^{-k})) = 0$ for each $i \in \{1,\ldots, 2^k-1\}$. Additionally, this argument shows that $\Phi_k$ is indeed an optimal bijection.  

The proof of this claim comprises the rest of Part I. 

\medskip
\paragraph*{Proof of Claim \ref{cl:positive-part} in Part I} 

First we deal with the case $P(p=q\in [2,\infty))$. Here we have:
\begin{align}
P &= \sum_{x\in X} t_i\norm{Q(x)}^p_p + (1-t_i)\norm{R(x)}^p_p \label{step-begin}\\
&= \sum_{x\in X} \frac{1}{t_i^{p-1}} \norm{x - \psi_i(x_{t_i})}^p_p + \frac{1}{(1-t_i)^{p-1}}\norm{\psi_i(x_{t_i}) - \Phi_k(x) }^p_p \nonumber \\
&= \sum_{x\in X} \frac{1}{(t_i - t_0)^{p-1}} \norm{\psi_0(x_0) - \psi_i(x_{t_i})}^p_p + \frac{1}{(t_{2^k}-t_i)^{p-1}}\norm{\psi_i(x_{t_i}) - \psi_1(x_1) }^p_p  \nonumber \\
&= \sum_{x\in X} \frac{1}{(t_i - t_0)^{p-1}} \medlnorm \sum_{j=1}^i \psi_{j-1}(x_{t_{j-1}}) - \psi_{j}(x_{t_j}) \medrnorm^p_p + \frac{1}{(t_{2^k}-t_i)^{p-1}}\medlnorm  \sum_{j=i+1}^{2^k} \psi_{j-1}(x_{t_{j-1}}) - \psi_{j}(x_{t_j}) \medrnorm^p_p \nonumber \\
&\leq  \sum_{x\in X} \frac{1}{(t_i - t_0)^{p-1}} \medlp \sum_{j=1}^i \norm{\psi_{j-1}(x_{t_{j-1}}) - \psi_{j}(x_{t_j}) }_p \medrp^p + \frac{1}{(t_{2^k}-t_i)^{p-1}}  \medlp \sum_{j=i+1}^{2^k} \norm{\psi_{j-1}(x_{t_{j-1}}) - \psi_{j}(x_{t_j}) }_p \medrp^p 
\label{step-triangle-ineq}
\\
&\leq  \sum_{x\in X} \sum_{j=1}^i  \frac{\norm{\psi_{j-1}(x_{t_{j-1}}) - \psi_{j}(x_{t_j}) }_p^p}{(t_j - t_{j-1})^{p-1}}  +  \sum_{j=i+1}^{2^k} \frac{\norm{\psi_{j-1}(x_{t_{j-1}}) - \psi_{j}(x_{t_j})}^p_p}{(t_j - t_{j-1})^{p-1}} 
\label{step-jensen-ineq}
\\
& =   \sum_{x\in X} \sum_{j=1}^{2^k}  \frac{\norm{\psi_{j-1}(x_{t_{j-1}}) - \psi_{j}(x_{t_j}) }^p_p}{(t_j - t_{j-1})^{p-1}} \nonumber \\
& =  2^{k(p-1)} \sum_{x\in X} \sum_{j=1}^{2^k}  \norm{\psi_{j-1}(x_{t_{j-1}}) - \psi_{j}(x_{t_j}) }^p_p \nonumber \\
& =  2^{k(p-1)}  \sum_{j=1}^{2^k} \sum_{x\in X} \norm{\psi_{j-1}(x_{t_{j-1}}) - \psi_{j}(x_{t_j}) }^p_p \nonumber \\
& =  2^{k(p-1)}  \sum_{j=1}^{2^k} \dlpp(\mu(t_{j-1}),\mu(t_j))^p 
\label{step-optimal-bij}
\\
&=  2^{k(p-1)}  \sum_{j=1}^{2^k} \lp \frac{\dlpp(X,Y)}{2^k}\rp^p =  \dlpp(X,Y)^p.
\end{align}

Step (\ref{step-triangle-ineq}) follows from the triangle inequality, (\ref{step-jensen-ineq}) follows from Lemma \ref{ineq:sturm-1}, and (\ref{step-optimal-bij}) follows because the $\psi_j$ maps are constructed using the $\ph_j$ maps, which are optimal by assumption. 

The case $P(p\in[2,\infty),q=2)$ follows almost immediately, as it is exactly analogous to the preceding case with $\norm{\cdot}_p^p$ terms replaced by $\norm{\cdot}_2^p$. Specifically, we get 
\[\dlp(X,Y)^p \leq P - CN +  f(\e) \leq \dlp(X,Y)^p - CN + f(\e).\]
But $\e>0$ was arbitrary, and $f(\e) \r 0$ as $\e \r 0$. So we have $\dlp(X,Y)^p \leq \dlp(X,Y)^p - CN$, and so $N = 0$.

Finally we handle the case $P(p\in (1,2),q=2)$. Here we have:
\begin{align}
P  -\e &=  (t_i - t_0) \medlnorm \widetilde{Q}(X^*)  \medrnorm_{l^p}^2 + (1-t_i)\medlnorm \widetilde{R}(X^*)  \medrnorm_{l^p}^2 \nonumber \\
&= \frac{t_i - t_0}{(t_i - t_0)^2}  \medlnorm \rho_0(X^*_{t_0}) - \rho_i(X^*_{t_i}) \medrnorm_{l^p}^2 + 
 \frac{1 - t_i }{(1 - t_i)^2}  \medlnorm \rho_i(X^*_{t_i}) - \rho_{2^k}(X^*_{t_{2^k}}) \medrnorm_{l^p}^2 \nonumber \\
&= \frac{1}{t_i - t_0}  \medlnorm \sum_{j=1}^i \rho_{j-1}(X^*_{t_{j-1}}) - \rho_j(X^*_{t_j}) \medrnorm_{l^p}^2 + 
 \frac{1}{1 - t_i}  \medlnorm \sum_{j=i+1}^{2^k} \rho_{j-1}(X^*_{t_{j-1}}) - \rho_{j}(X^*_{t_{j}}) \medrnorm_{l^p}^2 \nonumber \\
&\leq \frac{1}{t_i - t_0}   \lp \sum_{j=1}^i \medlnorm \rho_{j-1}(X^*_{t_{j-1}}) - \rho_j(X^*_{t_j}) \medrnorm_{l^p} \rp^2 + 
 \frac{1}{1 - t_i}  \lp \sum_{j=i+1}^{2^k}   \medlnorm \rho_{j-1}(X^*_{t_{j-1}}) - \rho_{j}(X^*_{t_{j}}) \medrnorm_{l^p} \rp^2 \label{step:case3-minkowski} \\
&\leq   \sum_{j=1}^i  \frac{ \medlnorm \rho_{j-1}(X^*_{t_{j-1}}) - \rho_j(X^*_{t_j}) \medrnorm_{l^p}^2}{t_j - t_{j-1}} + 
\sum_{j=i+1}^{2^k}   \frac{\medlnorm \rho_{j-1}(X^*_{t_{j-1}}) - \rho_{j}(X^*_{t_{j}}) \medrnorm_{l^p}^2}{t_j - t_{j-1}} \label{step:case3-jensen} \\
&=   \sum_{j=1}^{2^k}  \frac{ \medlnorm \rho_{j-1}(X^*_{t_{j-1}}) - \rho_j(X^*_{t_j}) \medrnorm_{l^p}^2}{t_j - t_{j-1}} \nonumber \\
&\leq 2^k \sum_{j=1}^{2^k}  \frac{\dlp(X,Y)^2}{2^{2k}}  + f(\e) \label{step:case3-optimal} \\
&= \dlp(X,Y)^2 + f(\e). 
\end{align}
Adjusting $f$ as needed, we write $P \leq \dlp(X,Y)^2 + f(\e)$. Here $f(\e) \r 0$ as $\e \r 0$. Since $\e> 0$ was arbitrary, it follows that $P \leq \dlp(X,Y)^2$.
Here Steps (\ref{step:case3-minkowski}), (\ref{step:case3-jensen}), and (\ref{step:case3-optimal}) hold by the Minkowski inequality for $l^p$ norms, by Lemma \ref{ineq:sturm-1}, and by the optimality of the $\phi_j$ maps, respectively. Note that the $\e$ error term comes from using the $\rho$ maps, which agree with the $\ph$ maps on off-diagonal points and incur an infinitesimal error from moving diagonal points. This concludes the proof of the claim.

By the discussion following the statement of Claim \ref{cl:positive-part}, and the discussion in \S\ref{sec:geom-idea}, we immediately obtain in the first two cases that $\dlpq(\g(t_i),\mu(t_i)) = 0$ for each $i \in \set{1,\ldots, 2^k - 1}$. In the third case, we obtain $\medlnorm \widetilde{Q}(X^*) - \widetilde{R}(X^*) \medrnorm_{l^p} = 0$ for a given $t_i$. This in turn implies that for each $ j \in \N$, we have $\rho_i(x^j_{t_i}) = (1-t_i)x^j + t\s(x^j)$. In the cases where $x^j$ or $\s(x^j)$ is off-diagonal, we then have $\psi_i(x^j_{t_i}) = \rho_i(x^j_{t_i})$. On the diagonal points of $X$ that get matched to diagonal points of $Y$, $\psi$ is the identity by definition. Thus we again have $\sum_{x \in X} \norm{x_{t_i} - \psi_i(x_{t_i})}_2^p = 0$, which shows $\dlpq(\g(t_i),\mu(t_i)) = 0$ in the case $q = 2$, $p \in (1,2)$.

\medskip
\paragraph*{Part II} 
We begin with an observation. Let $\Phi: A \r B$ be any optimal bijection between diagrams $A,B \in \dgml$ that have finitely many off-diagonal points. Any off-diagonal point of $A$ is mapped either to an off-diagonal point of $B$ or to a copy of its projection onto the diagonal in $B$. In particular, we know by optimality that $\Phi$ is the identity on each point on the diagonal of $A$ that is not the diagonal projection of a point in $B$. Since $A$ and $B$ both have finitely many off-diagonal points and hence finitely many diagonal projections, we know that $\Phi$ is the identity on all but finitely many points of $A$.

Now consider the sequence $\lp \Phi_{k} \rp_k$ of bijections $X\r Y$ chosen at the beginning of the proof. Let $X_1 \subseteq X$ denote the union of the finitely many off-diagonal points of $X$ with the finitely many copies of diagonal points that could possibly be matched to an off-diagonal point of $Y$ by projection. Define $Y_1 \subseteq Y$ similarly. We showed above that each $\Phi_k$ is optimal, so we know by the preceding observation that each $\Phi_k$ is the identity on $X\setminus X_1$. Thus we view each $\Phi_k$ as a map $X_1 \r Y$.

Write $X_1 = \{x_1,x_2,\ldots, x_n\}$. Choose a subsequence of $\lp \Phi_k \rp_k$ that is constant on $x_1$. Such a subsequence must exist by finiteness of $Y_1$. By choosing further subsequences, we obtain a subsequence that is constant on $X_1$. Let $\Phi_*:X \r Y$ denote the bijection given by this subsequence, and let $\g_*$ denote its induced geodesic. Fix $p,q$ in the prescribed ranges. Then we have $\dlpq(\mu(i2^{-k}),\g_*(i2^{-k})) = 0$ for each $i \in \{1,\ldots, 2^k-1\}$, for arbitrarily large $k$. Continuity of $\mu$ and $\g_*$ now shows that $\dlpq(\mu(t),\g_*(t)) = 0$ for each $t \in [0,1]$. \qedhere

\end{proof}

Finally we supply the proof of Theorem \ref{thm:char-II}.

\begin{proof}[Proof of Theorem \ref{thm:char-II}]
Part I of Theorem \ref{thm:char-I} is independent of any finiteness assumption, so it holds in this setting as well. Assume we are in the setup obtained from Part I of Theorem \ref{thm:char-I}, i.e. we have a sequence of optimal bijections $\Phi_k: X \r Y$. In particular, we have $C_p[l^q](\Phi_k) = \dlpq(X,Y)$ for each $k$. By Lemma \ref{lem:limiting-bijection}, we obtain a subsequence indexed by $L \subseteq \N$ and a limiting bijection $\Phi_*:X \r Y$ such that $\Phi_k \xrightarrow{k \in L, \, k \r \infty} \Phi_*$ pointwise and $C_p[l^q](\Phi_*) = \dlpq(X,Y)$.

Let $\g_*$ denote the geodesic induced by $\Phi_*$. Then we have $\dlpq(\mu(i2^{-k}),\g_*(i2^{-k})) = 0$ for each $i \in \{1,\ldots, 2^k-1\}$, for arbitarily large $k$. Continuity of $\mu$ and $\g_*$ now shows that $\dlpq(\mu(t),\g_*(t)) = 0$ for each $t \in [0,1]$. This proves the theorem.\qedhere
\end{proof}

\section{Discussion}

We have proved that in persistence diagram space equipped with several families of $l^p[l^q]$ metrics, every geodesic can be represented as a convex combination. The most interesting special case of this result is when $p =q=2$. The convex combination structure of geodesics in this case can be applied to obtain a variety of important geometric consequences, as shown in \cite{tmmh}. Several other cases remain open, e.g. the cases $p=q  \in (1,2)$ and $p \in (1,\infty),\, q \in (1,2)\cup (2,\infty)$. 

\medskip

\paragraph{Acknowledgements} We thank Facundo M\'{e}moli and Katharine Turner for many useful comments and suggestions.

\bibliographystyle{alpha}
\bibliography{../biblio}

\end{document}